\documentclass[11pt,leqno]{amsart}
\usepackage{amsmath,amsfonts,latexsym,graphicx,amssymb,url, color}
\usepackage{txfonts} 

\usepackage[hmargin=2.5 cm,vmargin=2.5 cm]{geometry}

\usepackage{color} 

\newcommand{\R}{{\mathbb R}} 

\newcommand{\Wo}{{\overset{o}{W}}}

\newcommand{\U}{{\mathbb U}}

\newcommand{\K}{{\mathbb K}}

\newcommand{\N}{{\mathbb N}}

\newcommand{\A}{\mathbb{A}}

\newcommand\norm[1]{\left\| #1\right\|}

\newcommand{\M}{{\mathcal M}}

\newcommand{\wei}[1]{\langle #1 \rangle}

\newcommand{\F}{{\mathbf F}}

\newcommand{\tA}{{\tilde{\A}}}

\newtheorem{theorem}{Theorem}[section]
\newtheorem{definition}[theorem]{Definition}
\newtheorem{remark}[theorem]{Remark}
\newtheorem{lemma}[theorem]{Lemma}
\newtheorem{proposition}[theorem]{Proposition}

\numberwithin{equation}{section}

\newcommand{\beq}{\begin{equation}}
\newcommand{\eeq}{\end{equation}}

\definecolor{darkred}{rgb}{.70,.12,.20}

\definecolor{darkgreen}{rgb}{.20,.52,.14}

\title[Weighted gradient estimates, degenerate Quasi-linear elliptic Equations] {Weighted Calder\'{o}n-Zygmund estimates for weak solutions of quasi-linear degenerate elliptic equations}
\author{Tuoc V. Phan}
\address{Department of Mathematics, University of Tennessee, Knoxville, 227 Ayres Hall, 1403 Circle Drive, Knoxville, TN 37996, U.S.A.}
\email{phan@math.utk.edu}
\begin{document}
\begin{abstract} This paper studies the Sobolev regularity estimates for weak solutions of a class of degenerate, and singular quasi-linear elliptic problems of the form $\textup{div}[\A(x,u, \nabla u)]= \textup{div}[\F]$ with non-homogeneous Dirichlet boundary conditions over bounded non-smooth domains. The coefficients $\A$ could be be singular,  and degenerate or both in $x$ in the  sense that they behave like some weight function $\mu$, which is in the $A_2$ class of Muckenhoupt weights. Global and interior weighted $W^{1,p}(\Omega, \omega)$-regularity estimates are established for weak solutions of these equations with some other weight function $\omega$. The results obtained are even new for the case $\mu =1$ because of the dependence on the solution $u$ of $\A$. In case of linear equations, our $W^{1,p}$-regularity estimates can be viewed as the Sobolev's counterpart of the  H\"{o}lder's regularity estimates  established by B. Fabes, C. E. Kenig, and R. P. Serapioni.
\end{abstract}

\maketitle
\section{Introduction and main results} \label{Intro-sec}
Let  $\Omega$ be a bounded domain in $\R^n$ with boundary $\partial \Omega$, and let $\mathbb{K}$ be an open interval in $\mathbb{R}$, or $\K =\R$. The theme of the paper is to establish the interior and global weighted $W^{1,p}$-regularity estimates for weak solutions of the non-homogeneous Dirichlet boundary value problem  
\begin{equation}  \label{main-eqn}
\left\{
\begin{array}{cccl}
\text{div} [\A(x,u, \nabla u)] & = & \text{div}[{\bf F}(x)]  &\quad \text{in} \quad \Omega,\\
  u & = & g(x) & \quad \text{on} \quad \partial\Omega,
 \end{array}
\right.
\end{equation}
where 
 ${\bf F}:\Omega \rightarrow \mathbb{R}^{n}$ is a given measurable vector field, $g: \Omega \rightarrow \mathbb{R}$ is a given measurable function. Moreover, the vector field function $\A: \Omega \times \mathbb{K} \times  \mathbb{R}^{n}~\rightarrow~\mathbb{R}^{n}$ is  a Carath\'{e}odory mapping satisfying the following natural coercivity, growth conditions
\begin{align} \label{Caratho-1}
& \A(x, \cdot, \cdot) \quad \text{is continuous  on} \ \mathbb{K} \times \mathbb{R}^n, \ \text{for almost every} \ x \in \Omega, \\
\label{Caratho-2}
& \A(\cdot, z, \eta) \quad \text{is measurable for each fixed} \ (z, \eta) \in \mathbb{K} \times  \mathbb{R}^n, \\
\label{up-ellip}
& |\A(x, z, \eta) | \leq \Lambda \mu(x) |\eta|, \quad \text{for almost every} \ x \in \Omega, \  \text{for all} \ (z, \eta) \in \mathbb{K} \times \mathbb{R}^n,\\ 
\label{lower-ellip}
& \Lambda^{-1}\mu(x) |\xi|^2 \leq \wei{\A(x, z, \xi), \xi},  \ \text{for almost every} \ x \in \Omega, \ \text{for all} \  (z, \xi) \in \mathbb{K} \times  \mathbb{R}^n,
\end{align}
where $\Lambda >0$ is a fixed constant, and $\mu :\mathbb{R}^n \rightarrow [0, \infty]$ is a weight function in the Muckenhoupt class $A_2$. 

We assume further that the vector field $\A$ is asymptotically Uhlenbeck in the following sense: There is a symmetric measurable matrix $\tilde{\A}: \Omega \rightarrow \mathbb{R}^{n \times n}$, and a bounded continuous function $\omega_0 : \overline{\mathbb{K}} \times [0, \infty) \rightarrow [0, \infty)$ such that
\begin{equation} \label{asymp-Uh}
|\A(x, z, \eta) - \tilde{\A}(x) \eta| \leq \omega_0 (z, |\eta|)\Big[1+ |\eta| \Big] \mu(x), \quad \text{for all most every} \quad x \in \Omega, \quad \forall \  (z,\eta) \in \mathbb{K} \times \mathbb{R}^n, 
\end{equation}
and 
\begin{equation} \label{omega-limit}
\lim_{s \rightarrow \infty} \omega_0(z, s) = 0, \ \text{uniformly in}\  z, \ \text{for} \  z \in \overline{\mathbb{K}} .
\end{equation}
Observe that from \eqref{up-ellip} and \eqref{lower-ellip}, we can assume also that 
\begin{equation} \label{ellip}
\Lambda^{-1} \mu(x) |\xi|^2 \leq \wei{\tilde{\A}(x) \xi, \xi} \leq \Lambda \mu(x) |\xi|^2, \quad \forall \ \xi \in \mathbb{R}^n, \quad \text{a.e.} \quad x \in \Omega.
\end{equation}
Our prototypical example is of the form
\[
\A(x,z,\eta) = a(x, z, |\eta|) \eta, \ \text{with} \  \lim_{s \rightarrow \infty} a(x, z, s) = \tilde{a}(x), \ \text{uniformly in} \ z \in \overline{\mathbb{K}}, \  
\text{and} \ \tilde{\A}(x) = \tilde{a}(x) \mathbb{I}_n.
\]

For a given weak solution $u \in W^{1,2}(\Omega, \mu)$ of \eqref{main-eqn}, the weighted Sobolev space with the weight $\mu$,  the main objective is to investigate the validity of the following  types of higher regularity weighted estimates 
\begin{equation} \label{main-est}
\int_{\Omega} |\nabla u|^p \omega(x) dx \leq C \left[\int_{\Omega}|\F/\mu|^p \omega(x) dx + \int_{\Omega} |\nabla g|^p \omega(x) dx + \omega(\Omega)\left\{\left(\frac{1}{\mu(\Omega)} \int_{\Omega}|\nabla u|^2\mu(x) dx \right)^{p/2} +1 \right\} \right]
\end{equation}
with $p > 2$, and some other weight function $\omega$ in some class of Muckenhoupt weights, whose definitions will be given later. 

The purpose of this paper is twofold. On one hand, it is the continuation of the developments of the recent work \cite{Bolegein, LTT, TN, NP} on the theory of Sobolev regularity theory for weak solutions of quasi-linear elliptic equations in which the coefficients $\A$ are dependent on the solution $u$. See, for instance \cite{AMP, BW-no, BW1, BW2, CP, CMP, CFL, Ragusa, Fazio-1, Kim-Krylov, KZ1, Krylov, Softova, MP-1, MP, M, Trud, Wang} for other work in the same directions but only for linear equations or for equations in which $\A$ is independent on $u$. On the other hand, this work includes the case that $\A$ could be singular or degenerate as a weight in some Munkenhoupt class of weights 
as considered in many papers such as \cite{Fabes, HKM, GW, MRW, MRW-1, Murthy-Stamp, NPS, Surnachev, Str} in which only Schauder's regularity of weak solutions are investigated. Moreover, even for the uniformly elliptic case, the results in this paper also improve those in \cite{Bolegein, LTT, TN, NP} since they do not require a-priori boundedness of weak solutions of \eqref{main-eqn}. This work also extends the recent work \cite{CMP} to the nonlinear case and two weighted estimates. Results of this paper can be considered as the Sobolev's regularity counterpart of the Schauder's one established in  \cite{Fabes, HKM, GW, MRW, MRW-1, Murthy-Stamp, NPS, Surnachev, Str} for singular, degenerate equations.

We need several notations and definitions before stating our results. For each weight function $\mu$ and $p \geq 1$, the weighted Lebesgue space  $L^p(\Omega, \mu)$ is defined as
\[
L^p(\Omega, \mu) = \left\{f: \Omega \rightarrow \R:  \norm{f}_{L^p(\Omega, \mu)}^p = \int_{\Omega} |f(x)|^p \mu(x) dx < \infty \right\}.
\]
As in \cite{Fabes}, if $\mu \in A_2$ and an open bounded domain $\Omega \subset \R^n$, we can define the weighted Sobolev space $W^{1,2}(\Omega, \mu)$ consisting of all measurable functions $f :\Omega \rightarrow \R$ such that $f \in L^2(\Omega, \mu)$ and all of their weak derivatives $\partial_{x_k}f \in L^2(\Omega, \mu)$ with $k =1, 2, \cdots, n$.  Moreover,
\[
\norm{f}_{W^{1,2}(\Omega, \mu)} = \norm{f}_{L^2(\Omega, \mu)} + \sum_{k=1}^n \norm{\partial_{x_k} f}_{L^2(\Omega, \mu)}.
\]
We also denote $W^{1,2}_0(\Omega, \mu)$ the closure of the compactly supported function space $C_0^\infty(\Omega)$ in $W^{1,2}(\Omega, \mu)$. For convenience in stating the results, we need a notation of the class of vector fields satisfying \eqref{Caratho-1}-\eqref{ellip}.
\begin{definition} Given an open set $\Omega \subset \mathbb{R}^n$, an open interval $\mathbb{K} \subset \mathbb{R}$, and the numbers $\Lambda >0, M_0 \geq 1$, and $M_1 >0$.  Let $\omega_0 : \K \times [0, \infty) \rightarrow [0, \infty)$ be continuous such that \eqref{omega-limit} holds and $\norm{\omega_0}_{\infty} \leq M_1$.  We denote $\mathbb{U}_{\Omega, \mathbb{K}}(\Lambda, M_0, M_1, \omega_0)$ be the set consisting all vector fields $\A: \Omega \times \mathbb{K} \times \mathbb{R}^n \rightarrow \mathbb{R}^n $ such that \eqref{Caratho-1}-\eqref{ellip} hold for some weight function $\mu \in A_2$  with $[\mu]_{A_2} \leq M_0.$ Moreover, with a given $\A \in \mathbb{U}_{\Omega, \mathbb{K}}(\Lambda, M_0, M_1, \omega_0)$, the matrix $\tA$ defined in \eqref{asymp-Uh} is called asymptotical matrix of $\A$.
\end{definition} \noindent
We now recall what it means by weak solutions.
\begin{definition} \label{weak-sol-def}\begin{itemize} 
\item[\textup{(i)}] For some $R >0$, let $B_R(y)$ be any ball of radius $R$ centered at  $ y \in \R^n$. Assume $\A \in \U_{B_{R}(y), \K} (\Lambda, M_0, M_1, \omega_0)$ with its asymptotical matrix $\tA$ and weight $\mu$.  Let $\F : B_R(y) \rightarrow \R^n$ be a vector field  such that $\F/\mu \in L^2(B_R(y), \mu)$. A function $u \in W^{1,2}(B_{R}(y), \mu)$ 
is said to be a weak solution of 
\[
\textup{div}[\A(x, u, \nabla u] = \textup{div}[\F], \quad \text{in} \quad B_{R}(y),
\]
if $u(x) \in \K$ for a.e. $x \in B_R(y)$, and 
\[
\int_{B_R(y)} \wei{\A(x, u, \nabla u), \nabla \varphi} dx = \int_{B_R(y)} \wei{\F,  \nabla \varphi} dx, \quad \forall \ \varphi \in C_0^\infty(B_R(y)).
\]
\item[\textup{(ii)}] Assume $\A \in \U_{\Omega, \K} (\Lambda, M_0, M_1, \omega_0)$ with its asymptotical matrix $\tA$ and weight $\mu$. Let $g \in W^{1,2}(\Omega, \mu)$, $\F : \Omega \rightarrow \R^n$ be a vector field  such that $\F/\mu \in L^2(\Omega, \mu)$. A function $u \in W^{1,2}(\Omega, \mu)$ is said to be a weak solution of \eqref{main-eqn} if $ u - g \in W^{1,2}_0(\Omega,\mu)$, $u(x) \in \K$ for a.e. $x \in \Omega$, and
\[
\int_{\Omega}\wei{\A(x, u, \nabla u), \nabla \varphi} dx = \int_{\Omega} \wei{\F, \nabla \varphi} dx, \quad \forall \ \varphi \in C_0^\infty(\Omega).
\]
\end{itemize}
\end{definition} \noindent
For any integrable function $f$ on a measurable $E \subset \mathbb{R}^n$,  we also denote the average of $f$ on $E$ as
\[
\wei{f}_{E} = \frac{1}{|E|} \int_{E} f(y) dy, \quad \text{with} \quad |E| = \int_{E} dx.
\]
Moreover, with a measurable set $E \subset \mathbb{R}^n$ and a positive Borel measure $\mu$, we also denote
\[
\mu(E) = \int_{E} d\mu(x).
\]
At this moment, we refer the readers to the definitions of $A_p$ classes of Muchkenhoupt weights and pairs of weights satisfying the $q$-Sawyer's condition in Definition \ref{A-p-weights}, and Definition \ref{Sawyer-pair}. Our first result is the interior estimates for the gradients of weak solutions of \eqref{main-eqn}.
\begin{theorem} \label{inter-theorem} Let $\Lambda >0, M_0, q \geq 1, M_1 >0, M_2 \geq 1$ and $p > 2$. Then there is $\delta = \delta (q, \Lambda, M_0, M_1, M_2, p, n) \in (0,1/8)$ and sufficiently small such that  the following holds: Suppose that $\K \subset \R$ is an open interval and $\omega_0:\K\times [0, \infty)\rightarrow [0~,~\infty)$ is continuous satisfying \eqref{omega-limit} and $\norm{\omega_0}_{\infty} \leq M_1$.  For some $R>0$, let $\A \in \mathbb{U}_{B_{2R}, \mathbb{K}} (\Lambda, M_0, M_1, \omega_0)$, with its asymptotical matrix $\tA$, and weight $\mu \in A_2$. Suppose that $\tA$ satisfies the following smallness condition on bounded mean oscillation with respect to the weight $\mu$:
\begin{equation}\label{BMO-B-R}
\sup_{0< r < R_0}\sup_{x \in B_{R}} \frac{1}{\mu(B_r(x))} \int_{B_{r}(x)} \Big| \tA(y) -\wei{\tA}_{B_r(x)} \Big|^2 \mu^{-1} (y) dy \leq \delta, \quad \text{for some} \quad R_0 \in (0, R].
\end{equation}
Then, if $u \in W^{1,2}(B_{2R}, \mu)$ 
is a weak solution of 
\[
\textup{div}[\A(x, u, \nabla u)] = \textup{div}[\F], \quad \text{in} \quad B_{2R},
\]
and if $\omega \in A_{q}$ with $[\omega]_{A_{q}} \leq M_2$, and the pair $(\mu, \omega)$ satisfies the $\frac{p}{2}$-Sawyer's condition, then there is a constant $C= C(p, q, \Lambda, M_0, M_1, M_2, \omega_0, n)>0$ such that the estimate
\begin{equation} \label{CZ-est}
\begin{split}
& \left(\frac{1}{\omega(B_R)}\int_{B_{R}} |\nabla u|^p \omega(x) dx \right)^{1/p} \\
& \leq C\left[ \left(
\frac{1}{\mu(B_{2R})}\int_{B_{2R}} |\nabla u|^2 \mu(x) dx \right)^{1/2} + \left(\frac{1}{\omega(B_{2R})}\int_{B_{2R}} \Big| \F/\mu\Big|^p \omega(x) dx \right)^{1/p}  +1\right],
\end{split}
\end{equation}
holds, assuming that $|\F/\mu| \in L^2(B_{2R}, \mu) \cap L^p(B_{2R}, \omega)$.
\end{theorem}
\noindent
We now have some discussion regarding \eqref{BMO-B-R}. We emphasize that in \eqref{BMO-B-R}, the average $\wei{\tA}_{B_r(x)}$ of $\tA$ in the ball $B_r(x)$ is taken with respect the Lebesgue measure.  The metric defined as in \eqref{BMO-B-R} is introduced in \cite{GC, GC-1, MW1, MW2} and called mean oscillation with respect to the weight $\mu$. A space consisting of all locally integrable functions satisfying the bounded mean oscillation with respect to some fixed given weight is also studied in \cite{GC, GC-1, MW1, MW2}. As noted in \cite{GC, GC-1, MW1, MW2}, this space is different from the usual weighted \textup{BMO} space, and is also different from the well-known John-Nirenberg \textup{BMO} space. We also note that the smallness condition \eqref{BMO-B-R} is natural as it reduces to the regular smallness condition in \textup{BMO} space that already used in known work \cite{AMP, BW-no, BW1, BW2, CP, CMP, CFL, Ragusa, Fazio-1, Kim-Krylov, KZ1, Krylov, Softova, MP-1, MP, M, Trud, Wang}. Moreover, as demonstrated by a counterexample in \cite{CMP}, the smallness condition \eqref{BMO-B-R} is necessary. In particular, as it is shown in \cite{CMP}, the estimate \eqref{CZ-est} is even not valid when the coefficients $\tA$ is uniformly continuous, but degenerate. In this light and compared to \cite{Fazio-2}, Theorem \ref{inter-theorem}, once reduced to the linear case, gives the right conditions on the coefficients so that the $W^{1,p}$-regularity estimates hold.

To derive the $L^p$-estimates for solutions up to the boundary, we need some regularity condition on $\partial \Omega$. The following definition is important in the paper.
\begin{definition} \label{Reifenberg-flatness} We say that $\Omega$ is a $(\delta, R_{0})$-Reifenberg flat domain if, for every $x\in \partial \Omega$ and every $r\in (0, R_{0})$, there exists a coordinate system $\{\vec{y}_{1}, \vec{y}_{2}, \cdots, \vec{y}_{n}\}$ which may depend on $x$ and $r$, such that in this coordinate system $x = 0$ and
\[
B_{r}(0) \cap \{y_{n} > \delta r\} \subset B_{r}(0) \cap \Omega \subset B_{r}(0) \cap \{y_{n} > -\delta r\}. 
\] 
\end{definition} \noindent
In the above and hereafter $B_{\rho}(x)$ denotes a ball of radius $\rho$ centered at $x$, $B_{\rho}^{+}(x)$ is the upper-half ball, and $\Omega_{\rho}(x) = B_{\rho}(x)\cap \Omega$, the portion of the ball in $\Omega$.  We also need the following additional definition in order to derive the local estimate of $\nabla u$ on the boundary.
\begin{definition} Let $\Omega \subset \R^n$ be a bounded domain. Let $R >0$ and $y_0 \in \overline{\Omega}$. The domain $\Omega_R(y_0) = \Omega\cap B_R(y_0)$ is said to be of type $(A,r)$,  if there is $A >0$ and $r \in (0, R/2)$ such that for every $y \in \Omega_R(y_0)$, and $\rho \in (0, r)$ if  $B_\rho(y) \cap\partial B_R(y_0)  \cap \overline{\Omega}\not= \emptyset$, and $B_\rho(y)\cap  \partial \Omega \cap \overline{B}_R(y_0) \not= \emptyset$ then 
\[
|\Omega_R(y_0) \cap B_\rho(y)| \geq A|B_\rho(y)|.
\]
\end{definition} \noindent
We remark that if $\overline{\Omega} \subset B_R(y_0)$, then $\Omega_R(y_0) =\Omega$ and it is a $(1,r)$ domain with $r \in (0, \text{dist}(\partial \Omega, B_R(y_0)))$. Moreover, 
if $\Omega = B_2^+(0)$, then $\Omega_1(0)$ is of type $(A, r)$ with some $A = A(n) >0$ and $ r\in (0,1)$. The following theorem is the main result of the paper.
\begin{theorem} \label{main-theorem}  Let $\Lambda >0, A>0, M_0, q \geq 1, M_1 >0, M_2 \geq 1$, and $p > 2$. Then there is a sufficiently small number $\delta \in (0,1/8)$ depending on  $p, q, \Lambda, A, M_0, M_1, M_2$, and $n$ such that  the following holds: Suppose that $\Omega$ is a $(\delta, R_0)$-Reifenberg flat domain for some $R_0 >0$, $\K \subset \R$ is an open interval, and $\omega_0:\K\times [0, \infty)\rightarrow [0~,~\infty)$ is continuous satisfying \eqref{omega-limit} and $\norm{\omega_0}_{\infty} \leq M_1$.   Suppose also that $\A \in \mathbb{U}_{\Omega, \mathbb{K}} (\Lambda, M_0, M_1, \omega_0)$, with its asymptotical matrix $\tA$, and weight $\mu \in A_2$ satisfying the following smallness condition on bounded mean oscillation with respect to 
the weight $\mu$:
\begin{equation} \label{smallness-BMO}
\sup_{0 < r < R_0} \sup_{x\in \Omega} \frac{1}{\mu(B_r(x))} \int_{B_{r}(x) \cap \Omega} \Big| \tA(y) -\wei{\tA}_{B_r(x) \cap \Omega} \Big|^2 \mu^{-1} (y) dy \leq \delta.
\end{equation}
Then, for every $R>0$, $y_0 \in \overline{\Omega}$ such that $\partial \Omega \cap B_R(y_0) \not=\emptyset$, and $\Omega_R(y_0)$ is of type $(A, r_0)$ with some small number $r_0>0$, and for $\omega \in A_{q}$ with $[\omega]_{A_q} \leq M_2$, and the pair $(\mu,\omega)$ satisfies the $\frac{p}{2}$-Sawyer's condition,  there is some positive constant $C= C(p, q, \Lambda, A, M_0, M_1, M_2, \Omega, R/r_0, \omega_0, n)$ such that if $u \in W^{1,2}(\Omega, \mu)$ is a weak solution of  \eqref{main-eqn}, the estimate
\[
\begin{split}
 \int_{\Omega_{R}(y_0)} |\nabla u|^p \omega(x) dx & \leq C \left[ \int_{\Omega_{2R}(y_0)} |\nabla g|^p \omega(x) dx +  \int_{\Omega_{2R}(y_0)} | \F/\mu|^p \omega(x) dx \right.\\
& \quad \quad + \left. \omega(\Omega_{R}(y_0)) \left\{\left(\frac{1}{\mu(B_{2R}(y_0))}\int_{\Omega_{2R}(y_0)} |\nabla u|^2\mu(x) dx \right)^{p/2}  +1 \right\}\right]
\end{split}
\]
holds if $|\nabla g| \in L^2(\Omega, \mu) \cap L^p(\Omega_{2R}(y_0), \omega)$, and $|\F/\mu| \in L^2(\Omega, \mu) \cap L^p(\Omega_{2R}(y_0), \omega)$.
\end{theorem}

A few comments on Theorem \ref{inter-theorem} and Theorem \ref{main-theorem} are now in ordered. Firstly, as we already discussed, the smallness condition \eqref{smallness-BMO} is optimal, as there are counterexamples given in \cite{CMP} for the degenerate case, and \cite{M} for the uniformly elliptic case. Secondly, we emphasize that the novelty in Theorem \ref{inter-theorem} and Theorem \ref{main-theorem} is that the coefficients $\A$ are non-uniformly elliptic and they could depend on $u$.  Moreover, unlike the previous work \cite{Bolegein, LTT, TN, NP}, Theorem \ref{inter-theorem} and Theorem \ref{main-theorem} do not assume the boundedness nor continuity of the weak solutions $u$. Therefore, even in case that $\A$ is uniformly elliptic, Theorem \ref{inter-theorem}, and Theorem \ref{main-theorem} are already new. Lastly, we would like to point out that the type of local $W^{1,p}$- estimates near the boundary as Theorem \ref{main-theorem} are necessary in many applications in which only local information of the data near the considered point on the boundary is known. Also, observe that Reifenberg flat domains are already used in previous work \cite{AMP, BW-no, BW1, BW2, MP-1, MP}. However, the local estimates near the boundary as stated in Theorem \ref{main-theorem} do not seem to appear, nor to be  produced directly from the global $W^{1,p}$-estimate results in these mentioned papers, even for smooth domains. Theorem \ref{main-theorem} bridges both versions of local and global weighted $W^{1,p}$-estimates for both smooth and non-smooth domains. To establish the local weighted $W^{1,p}$-estimates near the boundary $y_0 \in \partial \Omega$ for non-smooth Reifenberg flat domains, we impose the $(A,r_0)$-type condition on $\Omega_R(y_0)$. Several technical analysis issues are revised and improved to prove Theorem \ref{main-theorem}. 

There are two major difficulties in proving Theorems \ref{inter-theorem}-\ref{main-theorem}. The first one is due to the scaling properties the equations.
Observe that the class of equations \eqref{main-eqn} is not invariant under the usual scalings $u \rightarrow u/\lambda$ and the dilations $u (x) \rightarrow u(rx)/r$ with positive numbers $\lambda, r$. This is very serious since the theory of Sobolev regularity estimates for weak solutions relies heavily on these scalings and dilations, see for instance \cite{AMP, BW-no, BW1, BW2, CFL, Ragusa, Fazio-1, Kim-Krylov, KZ1, Krylov, Softova, MP-1, MP, M, Trud, Wang}. We overcome this by adapting  the perturbation technique with double-scaling parameter method introduced in \cite{LTT}, see also \cite{TN, NP} for the implementation of the method. In this perspective, the following observation regarding the scaling property of \eqref{main-eqn} is essential in the paper. 
\begin{remark} \label{remark-1} Let  $\A \in \U_{\Omega, \K}(\Lambda, M_0, M_1, \omega_0)$ with its corresponding asymptotical matrix $\tA$,  weight $\mu \in A_2$. Let $\lambda >0$ and define $\K_\lambda  = \K/\lambda$ and
\[ \A_{\lambda}(x,z,\eta) = \A(x, \lambda z, \lambda \eta)/\lambda, \quad (x,z, \eta) \in \Omega \times \K_\lambda \times \mathbb{R}^n.\]
Then, it is simple to check that 
\begin{equation} \label{lambda-asymp-Uh}
|\A_\lambda(x,z,\eta) - \tA(x) \eta| \leq \frac{1}{\lambda}\omega_0(\lambda z, \lambda \eta) (1 + |\lambda \eta|) \mu(x), \quad 
\forall \ (x,z,\eta) \in \Omega \times \K_\lambda \times \R^n.
\end{equation}
\end{remark} \noindent
The other major difficulty is due to the fact that we are working with two different weights, i.e. $\mu$ and $\omega$.  On one hand, all natural intermediate step estimates such as energy estimates, weighted maximal function operators $\M_\mu(|\nabla u|^2)$, and approximation estimates are performed in weighted spaces with weight $\mu$. On the other hand,  to obtain the estimates of $|\nabla u|$ in $L^p(\Omega, \omega)$, it requires to control the level sets, and density estimates of level sets of $\M_\mu(|\nabla u|^2)$ in measure $\omega$. Therefore, it requires to interchange the two measures. Lemma \ref{compare-omega-mu} below serves as the key ingredient for some kind of weak type $(1,1)$ estimates of two different weights. The lemma is indeed a simple consequence of the Coifman-Feferrman's result \cite{Coif-Feffer} on the reverse H\"{o}lder's inequality, and the doubling properties of the Muckenhoupt weights. 
Theorem \ref{inter-theorem}, and Theorem \ref{main-theorem} therefore can be viewed as a two weighted nonlinear Calder\'{o}n-Zygmund type estimates.

We conclude the section by highlighting the layout of the paper.  Some analysis preliminary tools in measure theories and weighted norm inequalities are reviewed in the next section, Section \ref{preliminaries}.  Section \ref{interior-approx-section} consists interior intermediate step estimates and the proof of Theorem \ref{inter-theorem}. Section \ref{global-regularity-section} treats the estimates near the boundary points and then proves Theorem \ref{main-theorem}.
\section{Preliminaries on weights and weighted norm inequalities} \label{preliminaries}
This section recalls several real analysis results, definitions needed in the paper. We first recall the definition of $A_p$- Muckenhoupt class of weights introduced in \cite{Muckenhoupt}.
\begin{definition} \label{A-p-weights} Let $1 \leq p < \infty$, a non-negative, locally integrable function $\omega :\R^n \rightarrow [0, \infty)$ is said to be in the class $A_p$ of Muckenhoupt weights if 
\[
\begin{split}
[\omega]_{A_p} & :=  \sup_{\textup{balls} \ B \subset \R^n} \left(\fint_{B} \omega (x) dx \right) \left(\fint_{B} \omega (x)^{\frac{1}{1-p}} dx \right)^{p-1} < \infty, \quad \textup{if} \quad p > 1, \\
[\omega]_{A_1} &: =  \sup_{\textup{balls} \ B \subset \R^n} \left(\fint_{B} \omega (x) dx \right)  \norm{\omega^{-1}}_{L^\infty(B)}  < \infty \quad \textup{if} \quad p  =1.
\end{split}
\]
\end{definition} \noindent
It turns out that the class of $A_p$-Muckenhout weights satisfies the reverse H\"{o}lder's inequality and the doubling properties. In particular, a measure of any $A_p$-weight is comparable with the Lebesgue measure. This is in fact a well-known result due to R. Coifman and C. Fefferman \cite{Coif-Feffer}, and it  is an important ingredient in the paper. 
\begin{lemma}[\cite{Coif-Feffer}] \label{doubling} For $1 < p < \infty$, the following statements hold true
\begin{itemize}
\item[\textup{(i)}] If $\mu \in A_{p}$,  then for every ball $B \subset \R^n$ and every measurable set $E\subset B$, 
\begin{equation*}
\mu(B) \leq [\mu]_{A_{p}} \left(\frac{|B|}{|E|}\right)^{p} \mu(E).
\end{equation*}
\item[\textup{(ii)}] If $\mu \in A_p$ with $[\mu]_{A_p} \leq M$ for some given $M \geq 1$, then there is $C = C(M, n)$ and $\beta = \beta(M, n) >0$ such that
\[
\mu(E) \leq C \left(\frac{|E|}{|B|} \right)^\beta \mu(B),
\]
 for every ball $B \subset \R^n$ and every measurable set $E\subset B$.
\end{itemize}
\end{lemma} \noindent
From Lemma \ref{doubling}, we can infer that any two Muckenhoupt weights are comparable. The following lemma is a consequence of Lemma \ref{doubling} and it is used frequently in this paper.
\begin{lemma} \label{compare-omega-mu} Let $1 < q, s < \infty$ and $M \geq 1, M' \geq 1$
Assume that $\mu \in A_s$ and $\omega \in A_q$ such that
\[
[\mu]_{A_s} \leq M, \quad [\omega]_{A_q} \leq M'.
\]
Then, there exists $\beta = \beta(M', n)>0$ such that for every ball $B$ and every measurable set $E \subset B$,
\[
\omega(E) \leq C(M, M', s, n) \left(\frac{\mu(E)}{\mu(B)} \right)^{\beta/s} \omega(B).
\]
\end{lemma}
\begin{proof}  By Lemma \ref{doubling} -(ii), there is $\beta = \beta(M', n) >0$ such that
\[
\frac{\omega(E)}{\omega(B)} \leq C(M', n) \left( \frac{|E|}{|B|} \right)^{\beta}.
\]
On the other hand, by Lemma \ref{doubling}-(i), we see that
\[
\left(\frac{|E|}{|B|} \right)^{s} \leq M \frac{\mu(E)}{\mu(B)}.
\]
Combining the two estimates, we obtain
\[
\frac{\omega(E)}{\omega(B)} \leq C(M, M', s, n) \left(\frac{\mu(E)}{\mu(B)} \right)^{\beta/s},
\]
as desired.
\end{proof} \noindent
Next, we state a standard result in measure theory.
\begin{lemma} \label{measuretheory-lp}
Assume that $g\geq 0$ is a measurable function in a bounded subset $U\subset \mathbb{R}^{n}$. Let $\theta>0$ and $N>1$ be given constants. If $\mu$ is a weight function in  $\mathbb{R}^{n}$, then for any $1\leq p < \infty$ 
\[
g\in L^{p}(U,\mu) \Leftrightarrow S:= \sum_{j\geq 1} N^{pj}\mu(\{x\in U: g(x)>\theta N^{j}\}) < \infty. 
\]
Moreover, there exists a constant $C>0$ depending only on $\theta, N$ and $p$.  such that 
\[
C^{-1} S \leq \|g\|^{p}_{L^{p}(U,\mu)} \leq C (\mu(U) + S). 
\]
\end{lemma} \noindent
Now, we discuss about the weighted Hardy-Littlewood maximal operator and its boundedness in weighted spaces. For a given locally integrable function $f$ and a weight $\mu$, we define the weighted Hardy-Littlewood maximal function as 
\[
\mathcal{M}_{\mu}f(x) = \sup_{\rho > 0}\fint_{B_{\rho}(x)}|f(y)| d(\mu(y)). 
\]
For functions $f$ that are defined on a bounded domain, we define 
\[
\mathcal{M}_{\mu, \Omega}f(x) = \mathcal{M}_{\mu}(f\chi_{\Omega})(x).
\]
We introduce the definition $p$-pair of weights satisfying Sawyer's condition. 
\begin{definition} \label{Sawyer-pair} Let $\mu, \omega$ be any two positive, locally finite Borel measures on $\R^n$ and let $1 < p < \infty$. The pair of measures $(\mu, \omega)$ is said to satisfy the p-Sawyer's condition if there is a constant $C >0$ such that
\[
\int_{B} \Big( \M_{\mu} (\chi_B \frac{d\sigma}{ d\mu} ) \Big)^p d\omega \leq C \sigma(B), \quad \forall \ \text{ball} \ B \subset \R^n,
\] 
where $\chi_B$ is the characteristic function of the ball $B$,  
\[
d\sigma = \Big(d\mu/d\omega \Big)^{p'} d\omega, \quad \text{and} \quad \frac{1}{p} + \frac{1}{p'} = 1.
\]
\end{definition}
\noindent
The following result is due to E. R. Sawyer in \cite{Sawyer}. Other proofs of this result can be found in \cite{Cruz-Uribe, Verbitsky}.
\begin{theorem} \label{Two-weighted-maximal-ineq}  Let $\mu, \omega$ be any two positive, locally finite Borel measures on $\R^n$ and let $1 < q < \infty$.  Then,
\[
\norm{\M_{\mu}}_{L^{q}(\R^n, \omega) \rightarrow L^{q}(\R^n, \omega) } \leq  C,
\]
if and only if the pair $(\mu, \omega)$ satisfies the $q$-Sawyer's condition.
\end{theorem} \noindent

We note that if $d\mu (x)= dx$, i.e. $\mu =1$, it follows from \cite{Hunt, Muckenhoupt}, see also \cite[Theorem 1, p. 201]{Stein}, that the pair $(1, \omega)$ satisfies the $q$-Sawyer's condition if and only if the weight function $\omega \in A_q$. Moreover, if $\mu$ is a doubling measure, the pair $(\mu, \mu)$ always satisfies the $q$-Sawyer's condition.
\noindent

We next state the following modified version of the Vitali's covering lemma, which is needed in the paper. The proof of 
this lemma can be found in the Appendix A at the end of the paper.
\begin{lemma}\label{Vitali} Suppose that $\Omega$ is a $(\delta, R_0)$-Reifenberg flat domain with some small numbers $\delta \in (0,1/8)$ and $R_0  >0$. Let  $ \omega \in A_{q}$ be a weight for some $q \in (1,\infty)$ with $[\omega]_{A_p} \leq M$, $R>0$, and $y_0 \in \overline{\Omega}$. Assume that  $\Omega_R(y_0)$ is of type $(A, r_0)$ with some  $r_0 \in (0,  \min\{R_0, R\}/2)$, $C, D$ are measurable sets satsifying $C \subset D \subset \Omega_R(y_0)$, and  there exists $0<\epsilon <1$ such that 
\begin{itemize}
\item[(i)]  $\omega(C) < \epsilon \omega(B_{r_{0}}(y)) $ for almost every $y\in \Omega_R(y_0)$, and 
\item[(ii)] for all $x\in \Omega_R(y_0)$ and $\rho \in (0, r_{0})$, if $\omega (C\cap B_{\rho}(x)) \geq \epsilon \omega(B_{\rho}(x))$, then $B_{\rho}(x)\cap \Omega_R(y_0) \subset D. $
\end{itemize}  
Then the estimate 
\[
\omega(C) \leq \epsilon'\omega(D), \quad \text{for} \quad \epsilon' =  \epsilon 5^{nq} \max\Big\{A^{-1}, 4^n \Big\}^q M^{2}.
\] 
\end{lemma}
A simple version of Lemma \ref{Vitali} with $\Omega = B_R$ is also used in the paper for the interior estimates, we therefore state it here for later reference. 
\begin{lemma}\label{Vitali-ball} Suppose $\omega \in A_q$ with some $1 < q < \infty$, and suppose that $R>0$, $C, D$ are measurable sets satisfying $C \subset D\subset B_R$. Assume also that there are $r_0  \in (0, R/2)$, and   $0<\epsilon <1$ such that 
\begin{itemize}
\item[(i)]  $\omega(C) < \epsilon \omega(B_{r_{0}}(y)) $ for almost every $y\in B_R$, and 
\item[(ii)] for all $x\in B_R(y_0)$ and $\rho \in (0, r_{0})$, if $\omega (C\cap B_{\rho}(x)) \geq \epsilon \omega(B_{\rho}(x))$, then $B_{\rho}(x)\cap B_R \subset D. $
\end{itemize}  
Then
\[
\omega(C) \leq \epsilon'\omega(D), \quad \text{for} \quad \epsilon' =  \epsilon 20^{nq}  [\omega]_{A_q}^{2}.
\] 
\end{lemma}

Finally, we state a remark which follows directly from Definition \ref{Reifenberg-flatness} about Reifenberg flat domains, see \cite[Remark 3.2]{MP} for details.
\begin{remark} \label{Rei-flat-remark} If $\Omega$ is a $(\delta, R_{0})$ flat domain with $\delta \in (0, 1/2)$, then for any point $x$ on the boundary and $ \rho \in (0, R_{0}/2)$,  there exists a coordinate system ${\vec{z}_1,\vec{z}_2,\cdots,  \vec{z}_n}$  with the origin at some point in the interior of  $\Omega$ such that in this coordinate system $x = -\delta \rho \vec{z}_n $ and
\[
B_{\rho}^{+}(0) \subset \Omega_{\rho}\subset B_{\rho}(0) \cap \{(z_{1}, \cdots, z_{n-1}, z_n): z_{n} > -4\delta \rho\}.\]
\end{remark}
\section{Interior $W^{1,p}$- regularity theory} \label{interior-approx-section}
%
\subsection{Interior approximation estimates} In this section, let $r >0$ and an open interval $\K \subset \R$. For $\A \in \U_{B_r, \K}(\Lambda, M_0, M_1, \omega_0)$, let $\A_\lambda$ be defined as in Remark \ref{remark-1} with some $\lambda >0$. We focus on the equation
\begin{equation} \label{eqn-in}
\text{div}[\A_\lambda(x, u, \nabla u)] = \text{div}[\F] , \quad \text{in} \quad B_r,
\end{equation}
We need to state what we mean by weak solution of \eqref{eqn-in}.
\begin{definition} A function $u \in W^{1,2}(B_r, \mu)$ is said to be a weak solution of \eqref{eqn-in} if $u(x) \in \K_\lambda:= \K/\lambda$ for all most every $x \in B_r$ and
\[
\int_{B_r} \wei{\A_\lambda(x, u, \nabla u), \nabla \varphi } dx = \int_{B_r} \wei{\F, \nabla \varphi} dx
\]
holds for all $\varphi \in C_0^\infty(B_r)$.
\end{definition} 
Let $\tA$ be the asymptotical matrix of $\A$ with the asymptotical weight function $\mu \in A_2$, and we recall that
\begin{equation} \label{ellip-interior}
\Lambda^{-1} |\eta|^2 \mu(x) \leq \wei{\tA(x) \eta, \eta} \leq \Lambda |\eta|^2 \mu(x), \quad \text{for a.e.} \ x \in B_r, \quad \forall \ \eta \in \R^n.
\end{equation}
and
\begin{equation} \label{M_0}
[\mu]_{A_2} \leq M_0.
\end{equation}
The following Proposition is the main result of the section.
\begin{proposition} \label{interio-approx-lemma} Let $\Lambda >0, M_0\geq 1, M_1 >0$ and a continuous function $\omega_0$ satisfying \eqref{omega-limit}, $\norm{\omega_0}_{\infty} \leq M_1$. Then, for every small number $\epsilon >0$, there exist $\delta = \delta(\epsilon, \Lambda, M_0, n) >0$ sufficiently small and  $\lambda =  \lambda(\epsilon, \Lambda, M_0, M_1, \omega_0, n) \geq 1$  such that the following holds:  Assume that $\A \in \U_{B_{r}(y), \K}( \Lambda, M_0, M_1, \omega_0)$, for some $r>0, y \in \mathbb{R}^n$, some open interval $\K \subset \mathbb{R}$, and with its asymptotical matrix $\tA$, and weight $\mu \in A_2$. 
Assume also that 
\[
\frac{1}{ \mu(B_r(y))}\int_{B_r(y)} |\tA(x) - \wei{\tA}_{B_r(y)}|^2 \mu^{-1} (x) dx \leq \delta.
\quad \text{and} \quad 
\frac{1}{ \mu(B_{7r/8}(y))} \int_{B_{7r/8}(y)} \Big| \frac{\F}{\mu} \Big|^2 \mu(x) dx  \leq \delta.
\]
Then, for every $\lambda \geq \lambda_0$, if  $u \in W^{1,2}(B_r(y), \mu)$ is a weak solution of
\[
\textup{div}[\A_\lambda(x, u, \nabla u)] = \textup{div}[\F], \quad \text{in} \quad B_r(y)
\]
satisfying
\[
\frac{1}{ \mu(B_{7r/8}(y))} \int_{B_{7r/8}(y)} |\nabla u|^2 \mu(x) dx \leq 1,
\]
then there is a Lipschitz function $v$ defined on $B_{3r/4}(y)$, and constant $C = C(\Lambda, M_0, n)$ such that
\[
\frac{1}{ \mu(B_{r/2}(y))}\int_{B_{r/2}(y)} |\nabla u - \nabla v|^2 \mu(x) dx \leq \epsilon,
\quad \text{and} \quad \norm{\nabla v}_{L^\infty(B_{r/2}(y))} \leq C .
\]
More precisely,  $v \in W^{1, 2}_{\textup{loc}}(B_{3r/4}(y))$ is a weak solution of the equation
\[
\textup{div}[\A_0 \nabla v]   =  0, \quad \text{in} \quad B_{\frac{3r}{4}}(y),
\]
for some constant symmetric matrix $\A_0$ satisfying $|\wei{\tA}_{B_r(y)} - \A_0| \leq \epsilon \frac{\mu(B_r(y))}{|B_r(y)|}.$
\end{proposition}
The rest of the section is to prove Proposition \ref{interio-approx-lemma}. By using translation $x \rightarrow x -y$, we can assume $y =0$. We split the procedure for the proof into two steps of approximations. 
 \subsubsection{First approximation estimates} Our first step is to 
approximate the solution $u$ of \eqref{eqn-in} by the solution $w$ of the following equation
\begin{equation} \label{w-interio}
\left\{
\begin{array}{cccl}
\text{div}[\tilde{\A}(x) \nabla w] & =& 0, &\quad \text{in}\quad B_{\frac{7r}{8}}, \\
w & =& u, & \quad \text{on} \quad \partial B_{\frac{7r}{8}}.
\end{array}\right.
\end{equation}
We note that for a given $u \in W^{1,2}(B_{7r/8}, \mu)$, it follows from \cite[Theorem 2.2]{Fabes} that there exists a unique weak solution $w \in W^{1,2}(B_{7r/8}, \mu)$ of \eqref{w-interio}. Our result of this section is the following weighted energy estimate for the weak solutions $u, w$.
\begin{lemma} \label{step-1-comparision} Assume that $u \in W^{1,2} (B_r)$ is a weak solution of \eqref{eqn-in} with some $\lambda >0$ and let $w \in W^{1,2}(B_{\frac{7r}{8}}, \mu)$ be 
the weak solution of \eqref{w-interio}. Then, there is a constant $C(\Lambda) >0$ such that
\begin{equation} \label{energy-w-step-1}
\fint_{B_{\frac{7r}{8}}} |\nabla w|^2 d\mu(x)   \leq C(\Lambda) \fint_{B_{\frac{7r}{8}}} |\nabla u|^2 d \mu (x).
\end{equation}
Moreover, for every $\delta >0$, there is $K_\delta >0$ depending only on $\delta$ and $\omega_0$ such that
\begin{equation} \label{u-w-step-1}
\fint_{B_{\frac{7r}{8}}} |\nabla (u-w)|^2 d \mu(x)  
\leq  C(\Lambda) \left[ \delta^2 \fint_{B_{\frac{7r}{8}}} |\nabla u|^2 d \mu(x) + \fint_{B_{\frac{7r}{8}}} \Big| \frac{\F}{\mu}\Big|^2d \mu(x)  + \frac{(M_1 K_\delta +\delta)^2}{\lambda^2} \right].
\end{equation}
\end{lemma}
\begin{proof} We write $\theta = {\frac{7r}{8}}$. Since $u-w \in W^{1,2}_0(B_\theta, \mu)$, we can use it as a test function for \eqref{w-interio} to 
obtain
\begin{equation} \label{w-energy-step-1}
\int_{B_\theta} \wei{\tilde{\A}(x) \nabla w, \nabla w} dx = \int_{B_\theta} \wei{\tilde{\A}(x) \nabla w, \nabla u}  dx.
\end{equation}
This and \eqref{ellip-interior} imply
\[
\begin{split}
\Lambda^{-1} \int_{B_\theta} |\nabla w|^2 \mu(x) dx & \leq \Lambda \int_{\theta} |\nabla u| |\nabla w| \mu(x) dx \\
& \leq \frac{\Lambda^{-1}}{2} \int_{B_\theta} |\nabla w|^2 \mu(x) dx + C(\Lambda) \int_{B_\theta }
|\nabla u|^2 \mu(x) dx.
\end{split}
\]
Therefore,
\begin{equation*} 
\int_{B_\theta } |\nabla w|^2 \mu(x) dx \leq C(\Lambda) \int_{B_\theta} |\nabla u|^2 \mu(x) dx.
\end{equation*}
This estimate proves \eqref{w-energy-step-1}.  Also, from \eqref{w-energy-step-1} it follows that
\[
\begin{split}
& \int_{B_\theta} \wei{\tilde{\A}(x) \nabla (u-w), \nabla (u-w)}dx\\
 & = \int_{B_\theta} \wei{\tilde{\A}(x) \nabla u, \nabla u} dx - 2 \int_{B_\theta} \wei{\tilde{\A}(x) \nabla u, \nabla w} dx + \int_{B_\theta} \wei{\tilde{\A}(x) \nabla w, \nabla w} dx \\
 & = \int_{B_\theta} \wei{\tilde{\A}(x) \nabla u, \nabla u} dx -  \int_{B_\theta} \wei{\tilde{\A}(x) \nabla u, \nabla w} dx = \int_{B_\theta} \wei{\tilde{\A}(x) \nabla u, \nabla u -\nabla w} dx .
\end{split}
\]
On the other hand, use $u-w$ as a test function for \eqref{eqn-in},  we obtain
\[
\int_{B_\theta} \wei{\A_\lambda(x,u, \nabla u), \nabla (u-w)} dx = \int_{B_\theta}\wei{ \F,  \nabla (u-w)} dx. 
\]
Combining the last two equalities, we obtain
\[
\begin{split}
& \int_{B_\theta} \wei{\tilde{\A}(x) \nabla (u-w), \nabla (u-w)}dx\\
& = \int_{B_\theta} \wei{\tilde{\A}(x) \nabla u - \A_\lambda(x,u,\nabla u), \nabla u -\nabla w} dx - \int_{B_\theta} \wei{\F , \nabla (u-w)}dx. 
\end{split}
\]
Therefore, it follows from \eqref{asymp-Uh}-\eqref{ellip}, and \eqref{lambda-asymp-Uh} that
\[
\begin{split}
& \Lambda^{-1} \int_{B_\theta} |\nabla (u-w)|^2 \mu(x) dx\\
& \leq \frac{1}{\lambda} \int_{B_\theta} \omega_0( \lambda u, \lambda |\nabla u|) (1+ \lambda |\nabla u|) |\nabla u - \nabla w| \mu(x)  dx + \int_{B_\theta} |\F||\nabla u - \nabla w| dx.
\end{split}
\]
It follows from \eqref{omega-limit} that for $\delta >0$, we can find a large number $K_\delta >0$ depending only on $\omega_0$ and $\delta$  such that $\omega_0(z, s) \leq \delta$ for all $s \geq K_\delta$ and for all $z \in \overline{\mathbb{K}} $. This and $\norm{\omega_0}_{\infty} \leq M_1$ in turn imply that
\[
\omega_0(z, s) (1+s) \leq  \delta (1+s) + M_1 K_\delta, \quad \forall s \geq 0, \quad \forall z \in \K,
\]
Hence,
\[
\begin{split}
& \Lambda^{-1} \int_{B_\theta} |\nabla (u-w)|^2 \mu(x) dx\\
& \leq \int_{B_\theta} \Big[\delta |\nabla u| +  \lambda^{-1}\Big (M_1 K_\delta + \delta\Big) \Big] |\nabla u - \nabla w|  \mu(x)  dx +  \int_{B_\theta} |\F|  |\nabla u - \nabla w| dx. 
\end{split}
\]
Then, by using H\"{o}lder's inequality and Young's inequality, we get
\[
\begin{split}
& \Lambda^{-1} \int_{B_\theta} |\nabla (u-w)|^2 \mu(x) dx\\
& \leq \epsilon \int_{B_\theta}|\nabla (u-w)|^2 \mu(x) +
  C(\epsilon) \left[ \delta^2 \int_{B_\theta} |\nabla u|^2 \mu(x) dx  +\lambda^{-2}\Big (M_1 K_\delta + \delta\Big)^2 \mu(B_\theta) + \int_{B_\theta} \Big| \frac{\F}{\mu}\Big|^2 \mu(x) dx   \right]. 
\end{split}
\]
Choosing $\epsilon < \Lambda^{-1}/2$, we then obtain
\[
 \int_{B_\theta} |\nabla (u-w)|^2 \mu(x) dx
\leq  C(\Lambda) \left[ \delta^2 \fint_{B_\theta} |\nabla u|^2 d\mu(x)  +\lambda^{-2}\Big (M_1 K_\delta + \delta\Big)^2+ \fint_{B_\theta} \Big| \frac{\F}{\mu}\Big|^2d \mu(x)    \right]  \mu(B_\theta) 
\]
This implies \eqref{u-w-step-1}. The proof is now complete.
\end{proof}
\subsubsection{Second approximation estimates} Our second approximation is the following lemma.
\begin{lemma} \label{L2-gradient-aprox} Let $\Lambda >0, M_0 > 1$.
 For every $\epsilon >0$ sufficiently small, there exists sufficiently small number $\delta' = \delta'(\epsilon, \Lambda, M_0, n)>0$  such that the following statement holds true: For some $r >0$, assume that \eqref{ellip-interior} and \eqref{M_0} hold and 
\[ 
 \frac{1}{ \mu(B_{r})}\int_{B_r} |\tilde{\A} -\langle\tilde{\A}\rangle_{B_r}|^{2} \mu^{-1} dx    \leq \delta',\]
then, for  every weak solution $w \in W^{1,2}(B_{7r/8}, \mu)$ of 
\[
\textup{div}[\tilde{\A} \nabla w] = 0, \quad \text{in} \quad B_{7r/8}
\]
satisfying $\fint_{B_{7r/8}} |\nabla w|^2 d\mu \leq C_0(\Lambda)$, there exists a constant matrix $\mathbb{A}_{0}$ and a weak solution $v \in W^{1,2}_{\textup{loc}}(B_{3r/4})$ of 
\[
\textup{div}[\A_0 \nabla v] = 0, \quad \text{in} \quad B_{3r/4}
\]
such that
\[
|\langle\mathbb{A}\rangle_{B_r} - \mathbb{A}_{0}| \leq \frac{\epsilon\mu(B_{r})}{|B_{r}|}  , \quad \text{and} \quad
\fint_{B_{r/2}} |\nabla w -\nabla v|^2 d\mu \leq \epsilon. 
\]
Moreover, there is $C =C(\Lambda, M_0, n)$ such that
\begin{equation} \label{L2-gradient-v-aprox}
\norm{\nabla v}_{L^\infty(B_{r/2})} \leq C.
\end{equation}
\end{lemma}
\begin{proof}  The proof is the same as that of \cite[Proposition 4.4]{CMP} with suitable scaling, we skip it.
\end{proof}
\subsubsection{Proof of Proposition \ref{interio-approx-lemma}}  We only need to glue Lemma \ref{step-1-comparision} and Lemma \ref{L2-gradient-aprox} together.  With given $\Lambda >0, M_0 \geq 1$ and $\epsilon >0$, let $\delta' = \delta'(\epsilon/2, \Lambda, M_0, n) >0$ be as in Lemma \ref{L2-gradient-aprox} and sufficiently small. Without loss of generality, we can assume that $\delta' \leq \epsilon/2$. Now, let $\delta =\delta(\epsilon, \Lambda, M_0, n): = \delta'/C_1(\Lambda, M_0,n) $, where $C_1(\Lambda, M_0,n) \geq 1$ is some constant which will be determined. Then, choose $\lambda_0 =\lambda_0 (\delta, M_1, \omega_0) = \lambda_0(\epsilon, \Lambda, M_0, M_1,\omega_0, n)$ sufficiently large such that
\begin{equation} \label{lambda-zero-choice}
\frac{M_1 K_\delta +\delta}{\lambda_0} \leq \delta,
\end{equation}
where $K_\delta$ is defined in Lemma \ref{step-1-comparision}. Now, assume that the assumptions in Proposition \ref{interio-approx-lemma} hold. Let $w$ be as in Lemma \ref{step-1-comparision}. Then, 
\[
\frac{1}{\mu(B_{7r/8})}\int_{B_{7r/8}} |\nabla w|^2 \mu(x) dx \leq C_0 (\Lambda) .
\]
Next, let $v$ be as in Lemma \ref{L2-gradient-aprox}.  It follows directly from Lemma \ref{L2-gradient-aprox} that
\[
\norm{\nabla v}_{L^\infty(B_{r/2})} \leq C(\Lambda, M_0, n), \quad 
\text{and} \quad \frac{1}{\mu(B_{r/2})}\int_{B_{r/2}} |\nabla w - \nabla v|^2 \mu(x) dx \leq \frac{\epsilon}{2}.
\]
From this and from Lemma \ref{step-1-comparision}, we obtain
\[
\begin{split}
& \frac{1}{\mu(B_{r/2})} \int_{B_{r/2}} |\nabla u - \nabla v|^2 \mu(x) dx \\
& \leq 
\frac{1}{\mu(B_{r/2})} \int_{B_{r/2}} |\nabla u - \nabla w|^2 \mu(x) dx + \frac{1}{\mu(B_{r/2})} \int_{B_{r/2}} |\nabla w - \nabla v|^2 \mu(x) dx \\
& \leq \frac{C(\Lambda)\mu(B_{7r/8})}{\mu(B_{r/2})} \left[\delta^2 \fint_{B_{7r/8}} |\nabla u|^2 d \mu(x)  + \fint_{B_{7r/8}} \Big| \frac{\F}{\mu} \Big|^2 d\mu(x)  + \frac{(M_1 K_\delta +\delta)^2}{\lambda^2} \right] + \frac{\epsilon}{2} . 
\end{split}.
\]
This together with \eqref{lambda-zero-choice} and the assumptions in the Proposition \ref{interio-approx-lemma} imply 
\[
\begin{split}
& \frac{1}{\mu(B_{r/2})}\int_{B_{r/2}} |\nabla u - \nabla v|^2 \mu(x) dx  \leq \frac{C(\Lambda) \mu(B_{7r/8})}{\mu(B_{r/2}}\delta^2 + \frac{\epsilon}{2}
\end{split}
\]
Moreover, observe that by the doubling property, there is $C(M_0, n) >0$ such that
\[
\frac{\mu(B_{7r/8})}{\mu(B_{r/2})} \leq C(M_0,n).
\]
Hence, there is $C_1(\Lambda, M_0, n) \geq 1$ such that
\[
\frac{1}{\mu(B_{r/2})}\int_{B_{r/2}} |\nabla u - \nabla v|^2 \mu(x) dx \leq C_1(\Lambda, M_0,n) \delta^2 + \frac{\epsilon }{2} .
\]
From this and the choice of $\delta$, we obtain
\[
\frac{1}{\mu(B_{r/2})}\int_{B_{r/2}} |\nabla u - \nabla v|^2 \mu(x) dx \leq \epsilon,
\]
and the proof is complete. 
\subsection{Interior level set estimates} This section consists of several lemmas preparing for the proof of Theorem \ref{inter-theorem}.  For $\lambda >0$, we consider the equation with the scaling parameter $\lambda$
\begin{equation} \label{lambda-B2R-eqn}
\textup{div}[\A_\lambda(x, u, \nabla u)] = \textup{div}[\F],  \quad \text{in} \quad B_{2R}.
\end{equation}
We begin with the following lemma.
\begin{lemma} \label{density-est-interior-1} Let $\Lambda >0, M_0 \geq 1$ be given. There exists $N = N(\Lambda, M_0, n) >1$ such that the following statement holds: Let $M_1 >0, M_2 \geq 1$, $\K \subset \R$ be some open interval, and let $\omega_0 : \K \times [0, \infty) \rightarrow [0, \infty)$ be continuous satisfying \eqref{omega-limit} and $\norm{\omega_0}_{\infty} \leq M_1$. Then, for every sufficiently small $\epsilon >0$, there exist sufficiently small $\delta = \delta(\epsilon, \Lambda, M_0, M_2, n) \in (0,1/8) $ and a large number $\lambda_0 = \lambda_0(\epsilon, \Lambda, M_0, M_1, M_2, \omega_0, n)\geq 1$ such that if $\A \in \U_{B_{2R}, \K}(\Lambda, M_0, M_1, \omega_0)$ with its asymptotical matrix $\tA$ and weight function $\mu \in A_2$ satisfying
\begin{equation} \label{interior-B2R-BM0-smallness}
\sup_{0 < \rho \leq R_0}\sup_{x \in B_{R}} \frac{1}{\mu(B_\rho(x))} \int_{B_{\rho}(x)} \Big| \tA(y) - \wei{\tA}_{B_{\rho}(x)}\Big|^2 \mu^{-1}(y) dy \leq \delta, \quad \text{for some} \ R_0 \in (0, R],
\end{equation}
then, for every $\lambda \geq \lambda_0$, every $u \in W^{1,2}(B_{2R}, \mu)$ a weak solution of \eqref{lambda-B2R-eqn}, every $y \in B_R$, and every $0 < r \leq R_0/3$, if  
\begin{equation} \label{non-empty-iterior-intersection}
B_r(y) \cap  \Big\{ B_{R}: \M_{\mu, B_{2R}}(|\nabla u|^2 ) \leq 1\Big \} \cap \Big \{B_{R}: \M_{\mu, B_{2R}} (|\F/\mu|^2) \leq \delta \Big \} \not= \emptyset,
\end{equation}
then 
\[
\omega(\{x \in B_{R}: \M_{\mu,B_{2R}}(|\nabla u|^2) >N\} \cap B_r(y)) < \epsilon \omega(B_r(y)),
\]
for every weight function $\omega \in A_q$ with $[\omega]_{A_q} \leq M_2$, for $1 \leq q <\infty$.
\end{lemma}
\begin{proof} For a given sufficiently small $\epsilon >0$, choose $\gamma \in (0,1)$ sufficiently small, to be determined, and depending only on $\epsilon, \Lambda, M_0$ and $M_2$. Let $\delta = \delta(\gamma, \Lambda, M_0, n) \in (0,1/8)$ and $\lambda_0 = \lambda_0 (\gamma, M_0, M_1, M_2, \omega_0, n)$ be defined as in Proposition \ref{interio-approx-lemma}. By \eqref{non-empty-iterior-intersection}, there is $x_0 \in B_r(y)$ such that
\begin{equation} \label{x-zero-interior}
\M_{\mu, B_{2R}} (|\nabla u|^2)(x_0) \leq 1, \quad \text{and} \quad \M_{\mu,B_{2R}}(|\F/\mu|^2) (x_0) \leq \delta.
\end{equation}
Observe that $B_{3r}(y) \subset B_{2R}$, and hence $u$ is a weak solution of
\begin{equation} \label{lambda-B3r-eqn-density-1}
\text{div}[\A_\lambda(x,u, \nabla u)] = \text{div}[\F], \quad \text{in} \quad B_{3r}(y).
\end{equation}
Moreover, since $B_{21r/8}(y) \subset B_{29r/8} (x_0) \cap B_{2R}$ and Lemma \ref{doubling}, we see that
\[
\begin{split}
\fint_{B_{21r/8}(y)} |\nabla u|^2 d\mu(x) & \leq \frac{\mu(B_{29r/8}(x_0))}{\mu(B_{21r/8}(y))} \frac{1}{\mu(B_{29r/8}(x_0))}\int_{B_{29r/8}(x_0) \cap B_{2R}} |\nabla u|^2 d\mu(x) \leq M_0\Big(\frac{29}{21}\Big)^{2n},\\
\fint_{B_{21r/8}(y)} |\F/\mu|^2 d\mu(x) & \leq \frac{\mu(B_{29r/8}(x_0))}{\mu(B_{21r/8}(y))} \frac{1}{\mu(B_{29r/8}(x_0))}\int_{B_{29r/8}(x_0) \cap B_{2R}} |\F/\mu|^2 d\mu(x) \leq M_0\Big(\frac{29}{21}\Big)^{2n} \delta.
\end{split}
\]
Moreover, because $3r \leq R_0$, with the assumption \eqref{interior-B2R-BM0-smallness} of the lemma, we see that
\[
\frac{1}{\mu(B_{3r}(y))}\int_{B_{3r}(y)} \Big| \tA(x) -\wei{\tA}_{B_{3r}(y)}\Big|^2 \mu^{-1} (x) dx \leq \delta.
\]
Now, let $u'(x) = u(x)/\Big[M_0\Big(\frac{29}{21}\Big)^{2n}\Big]$, $\F'(x)= \F(x)/\Big[M_0\Big(\frac{29}{21}\Big)^{2n} \Big]$, and $\lambda' = \lambda M_0\Big(\frac{29}{21}\Big)^{2n} \geq \lambda_0$, we infer from \eqref{lambda-B3r-eqn-density-1} that $u'$ is a weak solution of
\[
\text{div}[\A_{\lambda'}(x, u', \nabla u'] = \text{div}[\F'(x)], \quad x \in B_{3r}(y).
\]
Therefore, by applying Proposition \ref{interio-approx-lemma} for $u', F'$ and scaling back to $u, F$, we see that there exists $v$ such that
\begin{equation} \label{v-interior-compare-density}
\fint_{B_{3r/2}(y)}|\nabla u - \nabla v|^2 d\mu(x) \leq \gamma M_0\Big(\frac{29}{21}\Big)^{2n}, \quad 
\norm{\nabla v}_{L^{\infty}(B_{3r/2}(y))} \leq C_*(\Lambda, M_0, n).
\end{equation}
Now, let
\[
N = \max\{4C_*(\Lambda, M_0,n), 5^{2n} M_0 \}
\]
We claim that
\begin{equation} \label{set-est-interior}
\Big\{B_r(y): \M_{\mu, B_{3r/2}(y)}(|\nabla u - \nabla v|^2)  \leq C_*\Big\} \subset \Big\{ B_r(y): \M_{\mu, B_{2R}}(|\nabla u|^2) \leq N\Big\}.
\end{equation}
Indeed, let $x$ to be any point in the set on the left hand side of \eqref{set-est-interior}. We only need to show that
\begin{equation} \label{intereior-N-est}
\M_{\mu, B_{2R}}(|\nabla u|^2)(x) \leq N.
\end{equation}
Consider the ball $B_\rho(x)$. If $\rho \leq r/2$, we see that $B_{\rho}(x) \subset B_{3r/2}(y) \subset B_{2R}$. From this,  it follows
\[
\begin{split}
\frac{1}{\mu(B_\rho(x))} \int_{B_\rho(x)} |\nabla u|^2 \mu(x) dx & \leq 2 \left[ \frac{1}{\mu(B_\rho(x))} \int_{B_\rho(x)} |\nabla u - \nabla v|^2 \mu(x) dx + \frac{1}{\mu(B_\rho(x))} \int_{B_\rho(x)}|\nabla v|^2 \mu(x) dx \right] \\
& \leq 2\Big[\M_{\mu, B_{3r/2}(y)}(|\nabla u - \nabla v|^2)(x)  + \norm{\nabla v}_{L^\infty(B_{3r/2}(y)} \Big]\\
& \leq 2[C_* + C_*] = 4C_* \leq N.
\end{split}
\]
On the other hand, if $ \rho > r/2$, we see that $B_\rho(x) \subset B_{5\rho}(x_0)$. Hence, by applying Lemma \ref{doubling} -(i), and using \eqref{x-zero-interior}, we obtain
\[
\frac{1}{\mu(B_\rho(x))} \int_{B_\rho(x) \cap B_{2R}} |\nabla u|^2 \mu(x) dx \leq \frac{\mu(B_{5\rho}(x_0))}{\mu(B_\rho(x))}\frac{1}{\mu(B_{5\rho}(x_0))}\int_{B_{5\rho}(x_0) \cap B_{2R}}|\nabla u|^2 \mu(x) dx \leq M_05^{2n} \leq N.
\]
We therefore have proved \eqref{intereior-N-est}, which in turns proves \eqref{set-est-interior}. Observe that \eqref{set-est-interior} is equivalent to
\begin{equation} \label{compare-set-interior}
\Big\{B_r(y): \M_{\mu, B_{2R}}(|\nabla u|^2) > N\Big\} \subset E: = \Big\{B_r(y): \M_{\mu, B_{3r/2}(y)}(|\nabla u-\nabla v|^2) > C_*\Big\}.
\end{equation}
On one hand, by the weak type (1,1) estimate, the doubling property of $\mu$ in Lemma \ref{doubling}, and \eqref{v-interior-compare-density}, we see that
\[
\mu(E) \leq \frac{C(M_0) \mu(B_{3r/2}(y))}{C_*} \fint_{B_{3r/2}(y)} |\nabla u - \nabla v|^2 d\mu(x) \leq C'(\Lambda, M_0, n) \mu(B_r(y))\gamma.
\]
In other words,
\[
\frac{\mu(E)}{\mu(B_r(y))} \leq C'(\Lambda, M_0, n)\gamma.
\]
%
This and Lemma \ref{compare-omega-mu} imply
\[
\frac{\omega(E)}{\omega(B_r(y))} \leq C(M_0, M_2)\left( \frac{\mu(E)}{\mu(B_r(y))}\right)^{\beta/2} \leq C^*(\Lambda, M_0, M_2,n) \gamma^{\beta/2},
\]
for some constant $\beta = \beta(M_2, n)>0$. 
Hence, it follows from this and \eqref{compare-set-interior} that
\[
\begin{split}
\omega\Big(\Big\{B_r(y): \M_{\mu, B_{2R}}(|\nabla u|^2) > N\Big\} \Big) & \leq \omega(E) \leq C^*(\Lambda, M_0, M_2, n) \gamma^{\beta/2} \omega(B_r(y)).
\end{split}
\]
Now, if we choose $\gamma$ such that $\gamma^{\beta/2} C^*(\Lambda, M_0, M_2, n) = \epsilon$, the lemma follows.
\end{proof}
\noindent
We now can estimate a level set of $\M_{\mu, B_{2R}}(|\nabla u|^2)$. This is the main result in this subsection.
\begin{proposition} \label{good-lambda-interior} Let $\Lambda>0$, $M_k$, and $\omega_0, N$ be as in Lemma \ref{density-est-interior-1}, with $k = 1,2,3$. For any $\epsilon >0$, let $\delta \in (0,1/8)$ and $\lambda_0 \geq 1$ be defined as in Lemma \ref{density-est-interior-1}. Suppose that  $\A \in \U_{B_{2R}, \K}(\Lambda, M_0, M_1, \omega_0)$ with its asymptotical matrix $\tA$
and weight $\mu$ satisfying
\[
\sup_{0 < \rho \leq R_0}\sup_{x \in B_R} \frac{1}{\mu(B_\rho(x))} \int_{B_{\rho}(x)} \Big| \tA(y) - \wei{\tA}_{B_{\rho}(x)}\Big|^2 \mu^{-1}(y) dy \leq \delta, \quad \text{for some} \quad R_0 \in (0, R].
\]
Then, for every $\lambda \geq \lambda_0$ for $\omega \in A_q$ with $[\omega]_{A_q} \leq M_2, 1 < q < \infty$ and for every $u \in W^{1,2}(B_{2R}, \mu)$ a weak solution of \eqref{lambda-B2R-eqn}, if 
\[
\omega\Big(\Big\{B_R: \M_{\mu,B_{2R}}(|\nabla u|^2) > N \Big\} \Big) \leq \epsilon \omega(B_{R_0/6}(y)), \quad \forall \ y \in B_R,
\]
then for $\epsilon_1 = (20)^{qn} M_2^2\epsilon$,
\[
\begin{split}
& \omega\Big(\Big\{B_R: \M_{\mu, B_{2R}}(|\nabla u|^2) > N \Big\} \Big)  \leq \epsilon_1 \left[ \omega\Big(\Big\{B_R: \M_{\mu, B_{2R}}(|\nabla u|^2) > 1 \Big\} \Big) + \omega\Big(\Big\{B_R: \M_{\mu, B_{2R}}(|\F/\mu|^2) > \delta \Big\} \Big) \right].
\end{split}
\]
\end{proposition}
\begin{proof} Let us denote
\[
C = \Big\{B_R: \M_{\mu, B_{2R}}(|\nabla u|^2) > N \Big\},
\]
and
\[
D= \Big\{ B_R: \M_{\mu, B_{2R}}(|\nabla u|^2) >1 \Big\} \cup \Big\{ B_R: \M_{\mu, B_{R}}(|\F/\mu|^2) > \delta \Big\}
\]
Clearly, $C\subset D \subset B_R$. Our goal is to apply the modified Vitali's covering, Lemma \ref{Vitali-ball}, with $r_0 = R_0/6$. Clearly by the assumption of the lemma, $\omega(C) \leq \epsilon \omega(B_{r_0}(y))$ forall $y \in B_R$. Moreover, for some $\rho\in (0, r_0)$ and $y \in B_R$, 
if $\omega(C \cap B_\rho(y)) \geq \epsilon \omega (B_\rho(y))$, then it follows from Lemma \ref{density-est-interior-1} that
\[
B_\rho(y) \cap B_R \subset D.
\]
This verifies all conditions in the modified Vitali's covering lemma, Lemma \ref{Vitali-ball}. Therefore, Lemma \ref{good-lambda-interior} follows. 
\end{proof}
\subsection{Proof of Theorem \ref{inter-theorem}} By iterating Proposition \ref{good-lambda-interior}, we obtain the following result.
\begin{lemma} \label{iteration-interior-lemma} Let $\Lambda>0$, $M_k$, and $\omega_0, N$ be as in Lemma \ref{density-est-interior-1}, with $k = 1,2,3$. For any $\epsilon >0$, let $\delta \in (0,1/8), \lambda_0 \geq 1$ be defined as in Lemma \ref{density-est-interior-1}. Suppose that  $\A \in \U_{B_{2R}, \K}(\Lambda, M_0, M_1, \omega_0)$ with its asymptotical matrix $\tA$
and weight $\mu$ satisfying
\[
\sup_{0 < \rho \leq R_0}\sup_{x \in B_R} \frac{1}{\mu(B_\rho(x))} \int_{B_{\rho}(x)} \Big| \tA(y) - \wei{\tA}_{B_{\rho}(x)}\Big|^2 \mu^{-1}(y) dy \leq \delta, \quad \text{for some} \quad R_0 \in (0, R].
\]
Then, for every $\lambda \geq \lambda_0$ for $\omega \in A_q$ with $[\omega]_{A_q} \leq M_2$ with $1 < q < \infty$, and for every $u \in W^{1,2}(B_{2R}, \mu)$ a weak solution of \eqref{lambda-B2R-eqn}, if 
\[
\omega\Big(\Big\{B_R: \M_{\mu,B_{2R}} (|\nabla u|^2) > N \Big\} \Big) \leq \epsilon \omega(B_{R_0/6}(y)), \quad \forall \ y \in B_R,
\]
then for $\epsilon_1 = (20)^{qn} M_2^2\epsilon$,
\[
\begin{split}
 \omega\Big(\Big\{B_R: \M_{\mu, B_{2R}} (|\nabla u|^2) > N^k \Big\} \Big) & \leq \epsilon_1^k  \omega\Big(\Big\{B_R: \M_{\mu, B_{2R}}(|\nabla u|^2) > 1 \Big\} \Big) \\
 & \quad +  \sum_{i=1}^k \epsilon_1^i \omega\Big(\Big\{B_R: \M_{\mu, B_{2R}}(|\F/\mu|^2) > \delta N^{k-i} \Big\} \Big) .
\end{split}
\]
\end{lemma}
\begin{proof} We skip the proof because it is the same as that of Lemma \ref{global-iterating-lemma} below.
\end{proof} \noindent
After the accomplishment of Lemma \ref{iteration-interior-lemma}, the rest of the proof of Theorem \ref{inter-theorem} is similar to that of Theorem \ref{main-theorem}  in the next section.  We therefore skip it.
\section{Local boundary $W^{1,p}$-regularity theory and global $W^{1,p}$-regularity theory} \label{global-regularity-section}
To establish the global regularity estimates, we need both interior estimates and estimates up to the boundary. The interior theory are accomplished in previous sections. We now need to create the same theory up to the boundary.
\subsection{Boundary approximation estimates} \label{boundary-global-section} For $r >0$ and for $x_0~=~(x^{0}_1, x^{0}_2, \cdots, x^{0}_n)~\in~\mathbb{R}^n$, let us denote
\[
\begin{split}
B_r^+(x_0)  = \{ y = (y_1, y_2, \cdots, y_n) \in B_r(x_0):\ y_n >x^0_{n}\}, \quad B_r^+ = B_r^+(0), \\
T_r (x_0) = \{ x =  (x_1, x_2, \cdots, x_n) \in \partial B_{r}^+(0): \ x_n =x_{n}^{0}\}, \quad T_r = T_r(0).
\end{split}
\]
For given $x_0 \in \mathbb{R}^n$ and $\Omega \subset \R^n$, we also denote
\[
\Omega_r(x_0) = \Omega \cap B_r(x_0), \quad \partial_{w} \Omega_r(x_0) = \partial \Omega \cap B_r(x_0), \quad \Omega_r = \Omega_r(0).
\]
For your convenience, let us mention that the $(\delta, R_0)$-Reifenberg flat domains are defined in Definition \ref{Reifenberg-flatness}. In this section, for $r >0, \lambda >0$, and $\A \in \U_{\Omega_r, \K}(\Lambda, M_0, M_1, \omega_0)$, with some given $\omega_0$, we study the problem
\begin{equation} \label{boundary-eqn-r}
\left\{
\begin{array}{cccl}
\text{div}[\A_\lambda(x, u, \nabla u)]  & = & \text{div}[{\bf F}] & \quad \text{in} \quad \ \Omega_r, \\
u & =& g & \quad \text{on} \quad \partial_{w} \Omega_r.
\end{array}
\right.
\end{equation}
\noindent
The main result of the subsection is the following Proposition
\begin{proposition} \label{boundary-approximation-lemma} Let $\Lambda >0, M_0 \geq 1$, and $M_1 >0$.  Then, for every sufficiently $\epsilon >0$, there are $\delta= \delta(\epsilon, \Lambda, M_0, n) >0$ sufficiently small and $\lambda_0 = \lambda_0(\epsilon, \Lambda, M_0, M_1, \omega_0, n)~\geq ~1$ such that the following statement holds true: Assume that $\A \in \U_{\Omega_{r}, \K}(\Lambda, M_0, M_1, \omega_0)$, for some $r>0$, some open interval $\K \subset \mathbb{R}$, and with its corresponding asymptotical matrix $\tA$, and weight $\mu \in A_2$. Assume also that 
\[
B_r^+ \subset \Omega_r \subset B_r \cap \{ x_n > -4r \delta \}, \quad \frac{1}{\mu(B_r)}\int_{\Omega_{r}} | \tA(x) - \wei{\tA}_{\Omega_r}|^2 \mu^{-1}(x) dx \leq 
\delta,
\]
and
\[
\frac{1}{\mu(B_{7r/8})} \int_{\Omega_{7r/8}} \Big| \frac{\F}{\mu}\Big|^2 \mu(x) dx \leq \delta, \quad \frac{1}{\mu(B_{7r/8})} \int_{\Omega_{7r/8}}| \nabla g|^2 \mu(x) dx \leq \delta,
\]
and if $\lambda \geq \lambda_0$, and $u \in W^{1,2}(\Omega_r,\mu)$ is a weak solution of \eqref{boundary-eqn-r}
satisfying
 \[
 \frac{1}{\mu(B_{7r/8})} \int_{\Omega_{7r/8}} |\nabla u|^2 \mu(x) \leq 1,
\]
then, there is a function $v$ defined in $\Omega_{3r/4}$ and a constant $C = C(\Lambda, M_0, n)$ such that
\[
\frac{1}{\mu(B_{r/2})}\int_{\Omega_{r/2}} |\nabla u - \nabla v|^2 \mu(x) dx \leq \epsilon , \quad \text{and} \quad 
\norm{\nabla v}_{L^\infty(\Omega_{r/2})} \leq C.
\]
\end{proposition} \noindent
The remaining of the section is to prove Proposition \ref{boundary-approximation-lemma}. The proof is divided into two steps of approximations, Lemma \ref{step-1-comparision-bd} and Lemma \ref{L2-boundary-gradient-aprox} below. 
\subsubsection{Step 1: First approximation} We write $\theta = \frac{7r}{8}$. We first approximate the solution $u$ of \eqref{boundary-eqn-r} by and the corresponding homogeneous equation 
\begin{equation} \label{boundary-w-eqn}
\left\{
\begin{array}{cccl}
 \text{div}[\tilde{\A}(x) \nabla w ] &= &0 & \quad \text{in} \ \Omega_\theta,  \\ 
 w & = & u, & \quad \text{on} \ \partial \Omega_\theta\\ 
\end{array}
\right.
\end{equation}
We now state our definition of weak solutions for \eqref{boundary-eqn-r}, meanwhile weak solutions for \eqref{boundary-w-eqn} are understood by Definition \ref{weak-sol-def}.
\begin{definition} For $ g \in W^{1,2}(\Omega_r, \mu)$, and $\F \in L^2(\Omega_r, \mu^{-1})$, a function $u\in W^{1, 2}(\Omega_r, \mu)$ is a weak solution to \eqref{boundary-eqn-r} if $u -g \in \Wo^{1,2}(\Omega_r, \mu)$, $u(x) \in \mathbb{K}_\lambda: = \mathbb{K}/\lambda$ for a.e. $x \in \Omega_r$, and 
\[
\int_{\Omega_{r}}\langle\mathbb{A}_\lambda(x, u, \nabla u), \nabla \varphi \rangle dx  = \int_{\Omega_{r}} \langle {\bf F},\nabla \varphi\rangle dx,\quad \forall \varphi\in C^\infty_0(\Omega_{r}),
\]
where $\Wo^{1,2}(\Omega_r, \mu)$ denotes all functions $f \in W^{1,2}(\Omega_r, \mu)$ so that its zero extension to $B_r$ are in $W^{1,2}(B_r, \mu)$.
\end{definition} 
\begin{remark} \label{existence-remark} Given $u \in W^{1,2}(\Omega_r, \mu)$, it can be shown that $w$ is a  weak solution of \eqref{boundary-w-eqn} if and only if $\tilde{w} = w- u \in W^{1,2}_0(\Omega_\theta, \mu)$ is a weak solution of 
\[
\left\{
\begin{array}{cccl}
\textup{div}[\tA(x) \nabla \tilde{w}] & = & \textup{div}[\tA(x) \nabla u], & \quad \textup{in} \quad \Omega_\theta, \\
\tilde{w} & = & 0, & \quad \textup{on} \quad \partial \Omega_\theta.
\end{array}
\right.
\]
For the later, its existence and uniqueness of the weak solution $\tilde{w}$ is established in \cite[Theorem 2.2]{Fabes}.
\end{remark}
\noindent
The following energy estimates are the main result of our first approximation step.
\begin{lemma} \label{step-1-comparision-bd} Assume that $u \in W^{1,2}(\Omega_r, \mu)$ is a weak solution of \eqref{boundary-eqn-r} and $w \in W^{1,2}(\Omega_\theta,\mu)$ is a weak solution of \eqref{boundary-w-eqn}. Then, there is $C = C(\Lambda) >0$ such that
\begin{equation} \label{energy-w-bd}
\int_{\Omega_\theta} |\nabla w|^2\mu(x) dx  \leq C(\Lambda) \int_{\Omega_\theta} |\nabla u|^2 \mu (x) dx.
\end{equation}
Moreover, for every $\delta >0$, there is $K_\delta >0$ depending on $\delta$ and $\omega_0$ such that
\begin{equation} \label{u-w-bd}
\begin{split}
& \frac{1}{\mu(B_\theta)}\int_{\Omega_\theta} |\nabla (u-w)|^2 \mu(x) dx \\
& \leq  C(\Lambda, M_0)\left[ \frac{\delta^2}{\mu(B_\theta)}  \int_{\Omega_\theta}  |\nabla u|^2  \mu(x) dx  + \frac{1}{\mu(B_\theta)} \int_{\Omega_\theta} \Big| \frac{\F}{\mu}\Big|^2 \mu(x)  dx   + \frac{M_1^2 K_\delta^2 + \delta^2}{\lambda^2}\right].
\end{split}
\end{equation}
\end{lemma}
\begin{proof}  Since $w - u \in W^{1,2}_0(\Omega_\theta, \mu)$, we can use it as a test function for the equation \eqref{boundary-w-eqn} to obtain
\begin{equation} \label{boundary-w-test}
\int_{\Omega_\theta} \wei{\tilde{\A}(x) \nabla w, \nabla w} dx = \int_{\Omega_\theta} \wei{\tilde{\A}(x)\nabla w, \nabla u} dx. 
\end{equation}
From this, \eqref{ellip}, and the H\"{o}lder's inequality, we obtain
\[
\Lambda^{-1}\int_{\Omega_\theta} |\nabla w|^2 \mu(x) dx \leq \Lambda \left\{ \int_{\Omega_\theta} |\nabla u|^2 \mu(x) dx \right\}^{1/2} \left\{\int_{\Omega_\theta} |\nabla u|^2 \mu dx \right\}^{1/2}.
\]
Hence
\[
\int_{\Omega_\theta} |\nabla w|^2 \mu(x) dx \leq C(\Lambda) \int_{\Omega_\theta} |\nabla u|^2 \mu(x) dx.
\]
Therefore, we obtain \eqref{energy-w-bd}.  We next prove \eqref{u-w-bd}. Observe that it follows from \eqref{boundary-w-test} that
\[
\int_{\Omega_\theta} \wei{\tA(x) \nabla (u-w), \nabla (u-w) } dx = 
\int_{\Omega_\theta} \wei{\tA(x) \nabla u, \nabla (u-w)} dx.
\]
On the other hand, use $u-w \in W^{1,2}(\Omega_\theta, \mu)$ as a test function for the equation \eqref{boundary-eqn-r}, we see that
\[
\int_{\Omega_\theta} \wei{\A_\lambda(x, u, \nabla u), \nabla (u-w)} dx = \int_{\Omega_\theta} \wei{\F, \nabla (u-w)} dx.
\]
Combing the last two identity, we obtain
\[
\begin{split}
& \int_{\Omega_\theta} \wei{\tA(x) \nabla (u-w), \nabla (u-w) } dx = \int_{\Omega_\theta} \wei{\tA(x) \nabla u - \A_\lambda(x,u,\nabla u), \nabla (u-w)} dx + \int_{\Omega_\theta} \wei{ \F, \nabla (u-w)} dx.
\end{split}
\]
The proof now follows exactly the same as that of Lemma \ref{step-1-comparision} with the note that $\mu(B_\theta) \sim \mu(\Omega_\theta)$. We then skip it.
\end{proof}
\subsubsection{Step 2: Second approximation} We approximate the weak solution $w$ of \eqref{boundary-w-eqn} by the solution $v$ of the following equation on the flat domain
\begin{equation} \label{v-boundary-eqn}
\left\{
\begin{array} {cccl}
\textup{div}[\A_0 \nabla v] & = & 0, &\quad \text{in} \quad B_{\frac{3r}{4}}, \\
 v & = & 0, & \quad \text{on} \quad T_{\frac{3r}{4}}, 
\end{array} \right.
\end{equation}
for some constant, symmetric, uniformly elliptic matrix that is sufficiently close to $\wei{\tA}_{\Omega_r}$. Our next lemma is the main result of this subsection.
\begin{lemma} \label{L2-boundary-gradient-aprox} Let $\Lambda >0, M_0 \geq 1$ be fixed. For every $\epsilon >0$ sufficiently small, there exists sufficiently small number $\delta' >0$ depending on only $\epsilon, \Lambda, n$, and $M_0$ such that for every $\delta \in (0, \delta']$,  the following statement holds true: If \eqref{ellip} holds, $[\mu]_{A_2} \leq M_0$, 
\[
\frac{1}{\mu({B_r})} \int_{\Omega_r}| \tA - \wei{\tA}_{\Omega_r}|^2 \mu^{-1} dx \leq \delta^2,  \quad \frac{1}{\mu(B_{7r/8})}\int_{\Omega_{7r/8}} | \nabla g |^2  \mu (x) dx  \leq \delta^2, 
\]
 and if $w$ is a weak solution of  
\[
\left\{
\begin{array}{cccl}
\textup{div}[\tA(x) \nabla w] & = & 0, & \quad \text{in} \quad  \Omega_{\frac{7r}{8}}, \\
w & = & g, & \quad \text{on} \quad \partial_w \Omega_{\frac{7r}{8}},
\end{array}
\right.
\]
satisfying
\begin{equation} \label{L2-gradient-u-grad}
  \frac{1}{\mu(\Omega_{\frac{7r}{8}}) }\int_{\Omega_{\frac{7r}{8}}} |\nabla w|^2 d\mu \leq C_0 , \quad \text{with some} \quad C_0 = C_0(\Lambda, M_0) \geq 1,
\end{equation}
then 
there exists a constant, symmetric matrix $\mathbb{A}_0$ and a weak solution $v$ of \eqref{v-boundary-eqn} satisfying 
\begin{equation}\label{a_0-bar-a-4}
\Big| \langle \tA\rangle_{\Omega_r} - \mathbb{A}_0 \Big| \leq \frac{\epsilon  \mu(B_r)}{|B_r|}, 
\quad\text{and}\quad 
\frac{1}{\mu(B_{r/2})}\int_{\Omega_{r/2}} |\nabla w  -\nabla v|^2 \mu(x) dx \leq \epsilon.
\end{equation}
Moreover, there is $C =C(\Lambda, n, M_0)$ such that
\begin{equation*} 
\norm{\nabla v}_{L^\infty(\Omega_{r/2})} \leq C(\Lambda, n, M_0) .
\end{equation*}
\end{lemma}
\begin{proof}  Let us denote $g' = \frac{g}{C_0 }$, $w' = \frac{w - g}{C_0}$, and $\mathbf{F}'(x) =\tA(x) \nabla g'(x)$. Then, $w' \in W^{1,2}(\Omega_{\frac{7r}{8}}, \mu)$ is a weak solution of
\[
\left\{
\begin{array}{cccl}
\text{div}[\tA(x) \nabla w'] & = & - \text{div}[\F'], & \quad \text{in} \quad \Omega_{\frac{7r}{8}}, \\
 w'  & = & 0, & \quad \text{on} \quad \partial_w\Omega_{\frac{7r}{8}}.
 \end{array}
\right.
\]
Observe that 
\[
\frac{1}{\mu(B_{7r/8})} \int_{\Omega_{7r/8}} \Big| \frac{\F'}{\mu}\Big|^2 \mu(x) dx =  \frac{1}{\mu(B_{7r/8})}\int_{\Omega_{7r/8}} \Big|\frac{\tilde{\A}\nabla g'}{\mu(x)} \Big|^2 \mu(x) \leq
 \frac{\Lambda^2}{\mu(B_{7r/8})}\int_{\Omega_{7r/8}} |\nabla g'|^2 \mu(x) dx \leq \Lambda^2 \delta.
\]
Moreover,
\[
\frac{1}{\mu(B_{7r/8})} \int_{\Omega_{7r/8}} |\nabla w'|^2 \mu(x) dx \leq 1.
\]
The result then follows from \cite[Proposition 5.5]{CMP} by a suitable scaling.
\end{proof}
\subsubsection{Proof of Proposition \ref{boundary-approximation-lemma}}  We skip it because it is the same as that of Proposition \ref{interio-approx-lemma}.
\subsection{Level set estimates up to the boundary} To obtain the estimates of level sets for 
the maximal function $\M_\mu(|\nabla u|^2\chi_\Omega)$, we consider the following equation with the scaling parameter $\lambda>0$:
\begin{equation} \label{lambda-eqn-all}
\left\{
\begin{array}{cccl}
\textup{div}[\A_\lambda (x, u, \nabla u)] & =& \textup{div}[\F], & \quad \text{in} \quad \Omega. \\
u &= & g, & \quad \text{on}\quad \partial\Omega.
\end{array}
\right.
\end{equation}
We need several intermediate steps in order the prove the Proposition \ref{decay-est-global} below, which is the main result of the subsection. Our first lemma is in the same fashion as Lemma \ref{density-est-interior-1}.
\begin{lemma} \label{density-est-bd-1} Let $\Lambda >0, M_0 \geq 1$ be given. There exists $N_1 = N_1(\Lambda, M_0, n) >1$ such that the following statement holds: Let $M_1 >0, M_2 \geq 1$, $\K \subset \R$ be some open interval, and let $\omega_0 : \K \times [0, \infty) \rightarrow [0, \infty)$ be continuous satisfying \eqref{omega-limit} and $\norm{\omega_0}_{\infty} \leq M_1$. Then, for every sufficiently small $\epsilon >0$, there exist sufficiently small $\delta_1 = \delta_1(\epsilon, \Lambda, M_0, M_2, n) >0 $ and a large number $\lambda_1 = \lambda_1(\epsilon, \Lambda, M_0, M_1, M_2, \omega_0, n)\geq 1$ such that for every $R>0, y_0 \in \overline{\Omega}$ if $0 \in \Omega_R(y_0)$, and if
 $\A \in \U_{\Omega, \K}(\Lambda, M_0, M_1, \omega_0)$ with its asymptotical matrix $\tA$ and weight function $\mu \in A_2$ satisfying
\[
\frac{1}{\mu(B_{3r})} \int_{\Omega_{3r}} \Big| \tA(z) - \wei{\tA}_{\Omega_{3r}(0)}\Big|^2 \mu^{-1}(z) dz \leq \delta_1,
\]
for some $r \in (0,R/4)$, and 
\[
B_{3r}^+(0) \subset \Omega_{3r}(0) \subset B_{3r}(0) \cap \{(x', x_n) \in \R^n: x_n > - 12r\delta_1\}
\]
then for every $\lambda \geq \lambda_1$ and every weak solution $u \in W^{1,2}(\Omega, \mu)$ of \eqref{lambda-eqn-all}
satisfying
\begin{equation} \label{non-empty-b-intersection-1}
\begin{split}
& B_r \cap \Big\{\Omega_{R}(y_0): \M_{\mu, \Omega_{2R}(y_0)} (|\nabla u|^2) \leq 1 \Big \} \cap \\
& \quad \cap \Big \{\Omega_{R}(y_0): \M_{\mu, \Omega_{2R}(y_0)} (|\F/\mu|^2) \leq \delta_1 \Big \}  \cap \Big \{\Omega_{R}(y_0): \M_{\mu, \Omega_{2R}(y_0)} (|\nabla g|^2) \leq \delta_1 \Big \}\not= \emptyset,
\end{split}
\end{equation}
then it holds
\[
\omega\Big(\Big\{x \in \Omega_{R}(y_0): \M_{\mu,\Omega_{2R}(y_0)}(|\nabla u|^2) >N_1\Big\} \cap B_r\Big) < \epsilon \omega(B_r),
\]
for every weight function $\omega \in A_q$ with $[\omega]_{A_q} \leq M_2$, for $1 \leq q < \infty$.
\end{lemma}
\begin{proof} Though similar Lemma \ref{density-est-interior-1}, but details are different, hence a proof is needed. Let $\gamma >0$ sufficiently small to be determined and depending on $\epsilon, \Lambda, M_0, M_2, n$. Let $\delta_1 = \delta(\gamma, \Lambda, M_0, n) >0 $ and sufficiently small, and $\lambda_1 = \lambda_0(\gamma, M_0, M_1, \omega_0, n) \geq 1$, where $\delta$ and $\lambda_0$ are defined in Proposition \ref{boundary-approximation-lemma}.  By \eqref{non-empty-b-intersection-1}, there is $x_0 \in \Omega \cap B_r \cap B_R(y_0)$ such that
\begin{equation} \label{x-zero-assump}
\M_{\mu, \Omega_{2R}(y_0)}(|\nabla u|^2)(x_0) \leq 1, \quad \M_{\mu, \Omega_{2R}(y_0)} (|\F/\mu))(x_0) \leq \delta_1, \quad \M_{\mu, \Omega_{2R}(y_0)}(|\nabla g|^2)(x_0) \leq \delta_1.
\end{equation}
From this and since $\Omega_{21r/8} = B_{21r/8} \cap \Omega \subset \Omega_{29r/8}(x_0) = B_{29r/8}(x_0) \cap \Omega\subset  \Omega_{2R}(y_0)$, we obtain
\[
\begin{split}
\frac{1}{\mu(B_{21r/8})} \int_{\Omega_{21r/8}} |\nabla u|^2 \mu(x) dx \leq \frac{\mu(B_{29r/8}(x_0))}{\mu(B_{29r/8})} \frac{1}{\mu({B_{29r/8}(x_0)})} \int_{\Omega_{29r/8}(x_0)} |\nabla u|^2 \mu(x) dx \leq M_0 \Big(29/21\Big)^{2n}, \\
\frac{1}{\mu(B_{21r/8})} \int_{\Omega_{21r/8}} |\F/\mu|^2 \mu(x) dx \leq \frac{\mu(B_{29r/8}(x_0))}{\mu(B_{21r/8})} \frac{1}{\mu({B_{29r/8}(x_0)})} \int_{\Omega_{29r/8}(x_0)} |\F/\mu|^2 \mu(x) dx \leq M_0 \Big(29/21\Big)^{2n} \delta_1, \\
\frac{1}{\mu(B_{21r/8})} \int_{\Omega_{21r/8}} |\nabla g|^2 \mu(x) dx \leq \frac{\mu(B_{29r/8}(x_0))}{\mu(B_{21r/8})} \frac{1}{\mu({B_{29r/8}(x_0)})} \int_{\Omega_{29r/8}(x_0)} |\nabla g|^2 \mu(x) dx \leq M_0 \Big(29/21\Big)^{2n} \delta_1,
\end{split}
\]
Then, with suitable scaling, we can apply Proposition \ref{boundary-approximation-lemma} to find a function $v$ such that
\begin{equation} \label{bd-u-maximal-op}
\frac{1}{\mu(B_{3r/2})}\int_{\Omega_{3r/2}}|\nabla u - \nabla v|^2 \mu(x) dx \leq \gamma M_0  \Big(29/21\Big)^{2n}, \quad 
\norm{\nabla v}_{L^\infty(\Omega_{3r/2})} \leq C_* = C(\Lambda, M_0, n).
\end{equation}
Let $N_1 = \max\{4C_*, 5^{2n} M_0 \}$, we claim that
\begin{equation} \label{bd-set-compared}
 \Big\{x \in \Omega_r : \M_{\mu, \Omega_{3r/2}}(|\nabla u- \nabla v|^2) \leq C_*\Big\} \subset \Big\{x \in \Omega_r: \M_{\mu, \Omega_{2R}(y_0)} (|\nabla u|^2) \leq N_1  \Big\}.
\end{equation}
Indeed, let $x$ be a point in the set of the left hand side of \eqref{bd-set-compared}, we need to show that 
\begin{equation} \label{bd-N-Maximal-op}
\M_{\mu, \Omega_{2R}(y_0)}(|\nabla u|^2) (x) \leq N_1.
\end{equation}
 For $\rho >0$, and consider the case $\rho < r/2$. In this case, $\Omega_\rho(x) = \Omega \cap B_\rho(x) \subset \Omega_{3r/2} = \Omega \cap B_{3r/2}(0) \subset \Omega_{2R}(y_0)$. Therefore,
\[
\begin{split}
\frac{1}{\mu(B_\rho(x))} \int_{B_{\rho}(x) \cap \Omega_{2R}(y_0)} |\nabla u|^2 \mu(x) dx & \leq 2  \left[ \frac{1}{\mu(B_\rho(x))}  \int_{\Omega_{\rho}(x)}|\nabla u - \nabla v|^2 \mu(x) + \frac{1}{\mu(B_\rho(x))}  \int_{\Omega_{\rho}(x)} |\nabla v|^2 \mu(x) dx  \right] \\
&  \leq 2\left[\M_{\mu, \Omega_{3r/2}}(|\nabla u - \nabla v|^2)(x) + C_* \frac{\mu(\Omega_\rho(x))}{\mu(B_\rho(x))} \right] \\
& \leq 4C_* \leq N_1.
\end{split}
\]
Now, if $\rho \geq r/2$, then  $B_\rho(x) \subset B_{5\rho}(x_0)$. Therefore, it follows from \eqref{x-zero-assump} and Lemma \ref{doubling} that
\[
\begin{split}
\frac{1}{\mu(B_\rho(x))} \int_{B_\rho(x) \cap \Omega_{2R}(y_0)} |\nabla u|^2 \mu(x) & \leq \frac{\mu(B_{5\rho}(x_0))}{\mu(B_\rho(x)) } \frac{1}{\mu(B_{5\rho}(x_0))} \int_{B_{5\rho(x_0)} \cap \Omega_{2R}(y_0)} |\nabla u|^2 \mu(x) dx\\
& \leq \frac{\mu(B_{5\rho}(x_0))}{\mu(B_\rho(x)) }   \M_{\mu, \Omega_{2R}(y_0)}(\nabla u|^2) (x_0) \leq 5^{2n}M_0 \leq N_1,
\end{split}
\]
This estimates proves \eqref{bd-N-Maximal-op}, and therefore implies \eqref{bd-set-compared}. By taking the complement of both sets in  \eqref{bd-set-compared} in $\Omega_r$, we obtain
\begin{equation} \label{bdry-E-set}
 \Big\{x \in \Omega_r: \M_{\mu, \Omega_{2R}(y_0)} (|\nabla u|^2) < N_1  \Big\} \subset E: =  \Big\{x \in \Omega_r : \M_{\mu, \Omega_{3r/2}}(|\nabla u- \nabla v|^2) > C_*\Big\}.
\end{equation}
On the other hand, by the weak type (1,1)-estimate, the doubling property in Lemma \ref{doubling}, \eqref{bd-u-maximal-op},  we also have
\[
\mu(E) \leq \frac{C(M_0) \mu(B_{3r/2})}{C_*} \frac{1}{\mu(B_{3r/2})} \int_{\Omega_{3r/2}} |\nabla u - \nabla v|^2 \mu(x) dx \leq C'(\Lambda, M_0, n) \gamma \mu(B_r).
\]
Hence,
\[
\frac{\mu(E)}{\mu(B_r)} \leq C'(\Lambda, M_0, n) \gamma.
\]
From this, \eqref{bdry-E-set}, and Lemma \ref{compare-omega-mu}, there is $\beta = \beta(M_2, n)$ such that 
\[
\begin{split}
\frac{\omega\Big(\Big\{x \in \Omega_R(y_0): \M_{\mu, \Omega_{2R}(y_0)} (|\nabla u|^2) < N_1  \Big\} \cap B_r \Big)}{\omega(B_r)} 
& \leq \frac{\omega\Big(\Big\{x \in \Omega_r: \M_{\mu, \Omega_{2R}(y_0)} (|\nabla u|^2) < N_1  \Big\}}{\omega(B_r)} \\
& \leq C(M_0, M_2) \left(\frac{\mu(E)}{\mu(B_r)} \right)^{\beta/2} \\
& \leq C_0(\Lambda, M_0, M_2,n) \gamma^{\beta/2}.
\end{split}
\]
Hence,
\[
\omega\Big(\Big\{x \in \Omega_R(y_0): \M_{\mu, \Omega_{2R}(y_0)} (|\nabla u|^2) < N  \Big\} \cap B_r \Big) \leq C_0(\Lambda, M_0, M_2,n) \gamma^{\beta/2}.
\]
Therefore, by choosing $\gamma$ such that $C_0(\Lambda, M_0, M_2,n) \gamma^{\beta/2} = \epsilon$, the lemma follows.
\end{proof}
The following version of  Lemma \ref{density-est-interior-1} is needed, and we state it for later reference.
\begin{lemma} \label{density-est-interior-2} Let $\Lambda, M_0 \geq 1, M_1>0, M_2 \geq 1$ and sufficient $\epsilon >0$. There exist $N_2 = N_2(\Lambda, M_0, n) >1$  and sufficiently small $\delta_2 = \delta_2(\epsilon, \Lambda, M_0, M_2, n) >0 $ such that the following statement holds: Let $\K \subset \R$ be some open interval, and let $\omega_0 : \K \times [0, \infty) \rightarrow [0, \infty)$ be continuous satisfying \eqref{omega-limit} and $\norm{\omega_0}_{\infty} \leq M_1$,  there exists a large number $\lambda_2 = \lambda_2(\epsilon, \Lambda, M_0, M_1, M_2, \omega_0, n)\geq 1$ such that for every $y_0 \in \overline{\Omega}$ and $R>0$, if $\A \in \U_{\Omega_{2R}(y_0), \K}(\Lambda, M_0, M_1, \omega_0)$ with its asymptotical matrix $\tA$ and weight function $\mu \in A_2$ satisfying
\[
\sup_{0 < \rho < R_0} \sup_{x \in \Omega_{R}(y_0)} \frac{1}{\mu(B_{\rho}(x))} \int_{B_{\rho}(x)\cap \Omega} \Big| \tA(z) - \wei{\tA}_{B_{\rho}(x) \cap\Omega}\Big|^2 \mu^{-1}(z) dz \leq \delta_2, \quad \text{for some} \quad R_0>0,
\]
and if $\lambda \geq \lambda_2$,  and $u \in W^{1,2}(\Omega_{2R}(y_0), \mu)$ a weak solution of 
\eqref{lambda-eqn-all} 
so that with some $y \in \Omega_R(y_0)$, $0< r < \min\{R_0, R\}/3)$ such that $B_{3r}(y) \subset \Omega_{2R}(y_0)$, and  
\begin{equation} \label{non-empty-inter-intersection-all}
B_r(y) \cap \Big\{\Omega_R(y_0): \M_{\mu, \Omega_{2R}(y_0)}(|\nabla u|^2) \leq 1\Big \} \cap \Big \{\Omega_{R}(y_0): \M_{\mu, \Omega_{2R}(y_0)} (|\F/\mu|^2) \leq \delta_2 \Big \}  \not= \emptyset,
\end{equation}
then 
\[
\omega(\{x \in \Omega_{R}(y_0): \M_{\mu,\Omega_{2R}(y_0)}(|\nabla u|^2) >N_2\} \cap B_r(y)) < \epsilon \omega(B_\rho(y)),
\]
for every weight function $\omega \in A_q$ with $[\omega]_{A_q} \leq M_2$, for $1 \leq q \leq \infty$.
\end{lemma}
\begin{proof} The same as that of Lemma \ref{density-est-interior-1}, with $B_R$ replaced by $\Omega_R(y_0)$, and $B_{2R}$ replaced by $\Omega_{2R}(y_0)$.
\end{proof}
Combining Lemma \ref{density-est-bd-1} and Lemma \ref{density-est-interior-2}, we can prove the following lemma
\begin{lemma} \label{contra-lemma} Let $\Lambda >0, q \geq 1, M_{0} \geq 1, M_2 \geq 1, M_1 >0$ and let $\omega_0: \K \times [0, \infty) \rightarrow [0, \infty)$ be continuous, satisfy \eqref{omega-limit} and $\norm{\omega_0}_{\infty} \leq M_1$ with some open interval $\K \subset \R$. Then, for every $\epsilon >0$, there are $N = N(\Lambda, M_0, n)$, $\delta = \delta (\epsilon, q, \Lambda, M_0, M_2, n) \in (0, 1/8)$, and $\lambda_0 = \lambda_0(\epsilon, \Lambda, M_0, M_1, M_2, \omega_0, q, n) \geq 1$  such that the following statement holds: Suppose $\Omega$ is $(\delta, R_0)$-Reifenberg flat in $\R^n$ for some $R_0>0$, and suppose that $\A \in \U_{\Omega, \K}(\Lambda, M_0, M_1, \omega_0)$ with its asymptotical matrix $\tA$ and weight $\mu$ satisfying
\begin{equation} \label{BMO-bd-est-prop}
\sup_{0 < \rho < R_0}\sup_{x\in \Omega} \frac{1}{\mu(B_\rho(x))} \int_{\Omega_\rho(x)} \Big| \tA(y) - \wei{\tA}_{\Omega_{\rho}(x)} \Big|^2 \mu^{-1} (y) dy \leq \delta.
\end{equation}
 If $y_0 \in \overline{\Omega}$, $R>0$,  $\lambda \geq \lambda_0$, and $u\in W^{1,2}(\Omega, \mu)$ is a weak solution of \eqref{lambda-eqn-all} so that for some  $y\in \Omega_{R}(y_0)$, $0 < r < \min\{R_0, R\}/50$,
\begin{equation}\label{bdry-cond-0}
\omega\Big(\Big\{x\in\Omega_{R}(y_0): \M_{\mu, \Omega_{2R}(y_0)}(|\nabla u|^{2}) (x) > N\Big \}\cap  B_{r}(y)\Big) \geq \epsilon \omega(B_{r}(y))
\end{equation}
for some $\omega \in A_q$ with $[\omega]_{A_q} \leq M_2$, then
\[
\begin{split}
\Omega_{r}(y) &\subset  \Big \{\Omega_{R}(y_0): \M_{\mu, \Omega_{2R}(y_0)}(|\nabla u|^{2}) >1 \Big \}\cup  \\
& \quad \cup \Big \{\Omega_{R}(y_0): \M_{\mu, \Omega_{2R}(y_0)}(|\F/\mu|^{2}) > \delta \Big \} \cup \Big\{\Omega_{R}(y_0): \M_{\mu, \Omega_{2R}(y_0)}(|\nabla g|^{2}) > \delta^{2} \Big\}.
\end{split}
\]
\end{lemma}
\begin{proof} Let $N = \max\{N_1, N_2\}$,  $\epsilon' = \frac{\epsilon}{M_2 (17)^{qn}}$,  and 
\[
\begin{split}
& \delta = \min\{\delta_1(\epsilon', \Lambda, M_0, M_2, n), \delta_2(\epsilon, \Lambda, M_0, M_2, n), 1/8\}, \quad \text{and} \\
& \lambda_0  = \max\{\lambda_1(\epsilon', \Lambda, M_0, M_1, M_2, \omega_0, n), \lambda_2(\epsilon, \Lambda, M_0, M_1, M_2, \omega_0, n)\},
\end{split}
\]
where $N_1, \delta_1, \lambda_1$ are defined in Lemma \ref{density-est-bd-1}, and $N_2, \delta_2, \lambda_2$ are defined in Lemma \ref{density-est-interior-2}. We observe that since $\epsilon'$ depends on $q$, so do $\delta, \lambda_0$.

Since $y \in \Omega_{R}(y_0)$, and $0 < r < \min\{R_0, R\}/50$, we see that $B_{3r}(y) \subset B_{2R}(y_0)$. Therefore, if $B_{3r}(y) \cap \partial \Omega = \emptyset$, Lemma \ref{contra-lemma} follows directly by Lemma \ref{density-est-interior-2} and our choice of $N, \delta, \lambda_0$. It then remains to consider the case that $B_{3r}(y) \cap \partial \Omega \not=\emptyset$. In this case, 
we complete the proof by a contradiction argument. Assume there is $x_{0}\in \Omega_{r}(y)$ so that 
\[
\mathcal{M}_{\mu, \Omega_{2R}(y_0)}(|\nabla u|^{2})(x_{0}) \leq 1,\quad \mathcal{M}_{\mu, \Omega_{2R}(y_0)}\Big(|\F/\mu|^{2}\Big)(x_{0}) \leq \delta, \quad \text{and} \quad \mathcal{M}_{\mu, \Omega_{2R}(y_0)}\Big(|\nabla g|^{2}\Big)(x_{0}) \leq \delta.
\]
Because $B_{3r}(y) \cap \partial \Omega \not= \emptyset$, we can find $\tilde{y}_{0}\in \partial \Omega\cap B_{3r}(y)$. We observe that 
\[
x_{0} \in \Omega_r(y) = B_{r}(y)\cap \Omega \subset  B_{4r}(\tilde{y}_{0})\cap\Omega = \Omega_{4r}(\tilde{y}_0).
\]
Let $\rho = 7r$, we observe that $3\rho < R_0/2$. Since $\Omega$ is $(\delta, R_0)$ Reifenberg flat domain, and by Remark \ref{Rei-flat-remark}, there exists an orthonormal coordinate system $\{\vec{z}_1, \vec{z}_2,\cdots, \vec{z}_n\}$ in which $0 \in \Omega$,
\[
\tilde{y}_{0} = -3 \rho \delta \vec{z}_{n}\in \partial \Omega, \quad y_0 =z, \quad y = \hat{z},\quad x_{0} = z_{0},
\]
and 
\[
B_{3\rho}^{+}(0)\subset \Omega_{3\rho} \subset B_{3\rho} \cap \{z_{n} > -12\rho\delta\}. 
\]
In this new coordinate, the assumption \eqref{bdry-cond-0} becomes
\begin{equation} \label{bdr-cond-good-est-new-coordiente}
\omega\Big(\Big\{\Omega_{R}(z): \M_{\mu, \Omega_{2R}(z)}(|\nabla u|^{2}) > N\Big \}\cap  B_{r}(\hat{z})\Big) \geq \epsilon \omega(B_{r}(\hat{z})).
\end{equation}
We also observe that in the new coordinate system $z_0 \in B_{\rho}(0)$.  Indeed, 
\[ |\hat{z}| < |\hat{z} - \tilde{y}_0| + |\tilde{y}_0| = |y - \tilde{y}_0| + |\tilde{y}_0|  \leq 3r + 3\rho \delta , \]
and therefore 
\[
|z_{0}| \leq |z_0 -\hat{z}| + |\hat{z}| \leq 4r + 3\rho \delta  \leq 4r + \frac{3}{8} \rho < \rho.
\]
Collecting all estimates, after a change of coordinate system, and after a simple calculation, we obtain the followings
\begin{itemize}
\item[(i)] $u\in W^{1, 2}(\Omega, \mu)$ is a weak solution to \eqref{lambda-eqn-all}, $0 \in \Omega$,
\item[(ii)] $B_{3\rho}^{+}(0)\subset \Omega_{3\rho}\subset B_{3\rho}(0) \cap \{z= (z', z_n): z_{n}> -12\rho \delta\}$, 
\item[(iii)] $z_{0}\in B_{\rho}(0)\cap \Big\{\Omega_R(z): \M_{\mu, \Omega_{2R}(z)}(|\nabla u|^{2}) \leq 1 \Big\}$ and\\
$z_0 \in \Big\{\Omega_R(z): \M_{\mu, \Omega_{2R}(z)}(|\F/\mu|^{2}) \leq \delta \Big\} \cap \Big\{\Omega_R(z): \M_{\mu, \Omega_{2R}(z)}(|\nabla g|^{2}) \leq \delta \Big\} , and $
\item[(iv)] $B_{r}(\hat{z}) \subset B_{\rho}(0)\subset B_{17r}(\hat{z})$.
\end{itemize}
From \textup{(i)}-\textup{(iii)}, and the choice of $\delta$, we see that all the hypotheses of Lemma \ref{density-est-bd-1} are satisfied. We thus conclude that 
\[
\omega\Big(B_{\rho}(0)\cap\Big \{\Omega_R(z): \mathcal{M}_{\mu, \Omega_{2R}(z)}(|\nabla u|^{2}) > N \Big\} \Big) <  \epsilon'  \omega(B_{\rho}(0)). 
\]
Moreover, from item $\textup{(iv)}$, we infer that 
\[
\begin{split}
\omega\Big(B_r(\hat{z}) \cap \Big\{\Omega_{R}(z): \M_{\mu, \Omega_{2R}(z)}(|\nabla u|^{2}) > N \Big\}  \Big) &\leq\omega\Big (B_{\rho}(0)\cap \Big \{\Omega_R(z): M_{\mu, \Omega_{2R}(z)} |\nabla u|^{2}) > N \Big\} \Big)\\
&<  \frac{\epsilon }{M_{2}17^{qn}}\omega(B_{\rho}(0)) \leq \frac{\epsilon }{M_{2}17^{qn}}\omega(B_{17r}(\hat{z})) \\
&\leq  \frac{\epsilon }{M_{2}17^{qn}}M_{2}17^{qn} \mu(B_{r}(\hat{z})) \\
&= \epsilon\mu(B_{r}(\hat{z})),
\end{split}
\]
where we have used the doubling property of the $\mu$ as in Lemma \ref{doubling}. The last estimate obviously contradicts \eqref{bdr-cond-good-est-new-coordiente}, and thus the proof is complete. 

\end{proof}
Finally, we can estimate the density of the level sets of $\M_{\mu, \Omega_{2R}(y_0)}(|\nabla u|^2)$, which is the main result of the subsection.
\begin{proposition} \label{decay-est-global} Let $A>0, q \geq 1, \Lambda >0, M_{0} \geq 1, M_2 \geq 1, M_1 >0$, and  let $\epsilon >0$ sufficiently small. Also, let  $\omega_0: \K \times [0, \infty) \rightarrow [0, \infty)$ be continuous, satisfy $\norm{\omega_0}_{\infty} \leq M_1$ with some open interval $\K \subset \R$, and then let $N = N(\Lambda, M_0, n) \geq 1$,  $\delta = \delta (\epsilon, q, \Lambda, M_0, M_2, n) \in (0,1/8)$, and $\lambda_0 = \lambda_0(\epsilon, \Lambda, M_0, M_1, M_2, \omega_0, q, n) \geq 1$ be as in Lemma \ref{contra-lemma}.  Suppose that $\Omega$ is $(\delta, R_0)$-Reifenberg flat in $\R^n$ for some $R_0>0$, and $\A \in \U_{\Omega, \K}(\Lambda, M_0, M_1, \omega_0)$ with its asymptotical matrix $\tA$ and weight $\mu$ satisfying
\[
\sup_{0 < \rho < R_0}\sup_{x\in \Omega} \frac{1}{\mu(B_\rho(x))} \int_{\Omega_\rho(x)} \Big| \tA(y) - \wei{\tA}_{\Omega_{\rho}(x)} \Big|^2 \mu^{-1} (y) dy \leq \delta.
\]
Then, for any $\lambda \geq \lambda_0$, if $u\in W^{1,2}(\Omega, \mu)$ is a weak solution of \eqref{lambda-eqn-all} such that for some  $R>0, y_0 \in \overline{\Omega}$, some fixed $0 < r_0 < \min\{R_0, R\}/50$, and some $\omega \in  A_q$ satisfying $[\omega]_{A_q} \leq M_2$,
\begin{equation}\label{bdry-cond-1}
\omega(\{\Omega_R(y_0): \M_{\mu, \Omega_{2R}(y_0)}(|\nabla u|^{2})  > N \}) \leq \epsilon \omega(B_{r_0}(y)), \quad \forall \ y \in \overline{\Omega}_{R}(y_0),
\end{equation}
and $\Omega_R(y_0)$ is of $(A, r_0)$ type, then with $\epsilon_1 = 5^{nq} \max\{A^{-1}, 4^n\}^{q} M_2^2\epsilon$,
\[
\begin{split}
& \omega\Big (\Big\{\Omega_R(y_0): \M_{\mu, \Omega_{2R}(y_0)}(|\nabla u|^2 > N \Big\} \Big)  \\
& \leq \epsilon_1\left[ \omega\Big (\Big\{\Omega_R(y_0): \M_{\mu, \Omega_{2R}(y_0)}(|\nabla u|^2 )> 1 \Big\} \Big) \right.\\
& \quad + \left. \omega\Big (\Big\{\Omega_R(y_0): \M_{\mu, \Omega_{2R}(y_0)}(|\nabla g|^2) > \delta  \Big\} \Big)+\omega\Big (\Big\{\Omega_R(y_0): \M_{\mu, \Omega_{2R}(y_0)}(|\F/\mu|^2) > \delta \Big\} \Big) \right].
\end{split}
\]
\end{proposition}
\begin{proof} 
Let us also denote
\[
C= \Big\{\Omega_R(y_0): \M_{\mu, \Omega_{2R}(y_0)}(|\nabla u|^2) > N \Big\},
\]
and 
\[
\begin{split}
D &= \Big \{\Omega_R(y_0): \M_{\mu, \Omega_{2R}(y_0)}(|\nabla u|^2) > 1 \Big\} \cup \\
& \quad \quad \cup \{\Omega_R(y_0): \M_{\mu, \Omega_{2R}(y_0)}(|\nabla g|^2) > \delta  \Big\}  \cup \Big \{\Omega_R(y_0): \M_{\mu, \Omega_{2R}(y_0)}(|\F/\mu|^2) > \delta \Big\}.
\end{split}
\]
Clearly $C \subset D\subset \Omega_R(y_0)$. Moreover, by the assumption, we see that $\omega(C) < \epsilon \omega(B_{r_0}(y))$ for all $y \in \overline{\Omega}_R(y_0)$. Therefore, (i) of Lemma \ref{Vitali} holds. Moreover, it is simple to check that (ii) follows from Lemma \ref{contra-lemma}. Therefore, Lemma \ref{decay-est-global} is just a consequence of Lemma \ref{Vitali}.
\end{proof}
\subsection{Proof of the global weighted $W^{1,p}$-regularity estimates}
From the Proposition \ref{decay-est-global} and an iterating procedure, we obtain the following lemma
\begin{lemma} \label{global-iterating-lemma} Let $A, \Lambda >0, M_{0}, M_2, q \geq 1, M_1 >0$, and  let $\epsilon >0$ sufficiently small. Also, let  $\omega_0: \K \times [0, \infty) \rightarrow [0, \infty)$ be continuous, satisfy $\norm{\omega_0}_{\infty} \leq M_1$ with some open interval $\K \subset \R$, and let $N = N(\Lambda, M_0, n)$,  $\delta = \delta (\epsilon, q, \Lambda, M_0, M_2, n) \in (0,1/8)$, and $\lambda_0 = \lambda_0(\epsilon, \Lambda, M_0, M_1, M_2, \omega_0, q, n)$ be as in Proposition \ref{decay-est-global}.  Suppose that $\Omega$ is $(\delta, R_0)$-Reifenberg flat in $\R^n$ for some $R_0>0$, and $\A \in \U_{\Omega, \K}(\Lambda, M_0, M_1, \omega_0)$ with its asymptotical matrix $\tA$ and weight $\mu$ satisfying
\[
\sup_{0 < \rho < R_0}\sup_{x\in \Omega} \frac{1}{\mu(B_\rho(x))} \int_{\Omega_\rho(x)} \Big| \tA(y) - \wei{\tA}_{B_{\rho}(x)} \Big|^2 \mu^{-1} (y) dy \leq \delta.
\]
Then, for any $\lambda \geq \lambda_0$, $y \in \overline{\Omega}$ and $R>0$ such that $\Omega_R(y_0)$ is of type $(A, r_0)$ for some fixed $0 < r_0 < \min\{R_0, R\}/50$, if $u\in W^{1,2}(\Omega, \mu)$ is a weak solution of \eqref{lambda-eqn-all} such that  for some $\omega \in  A_q$ satisfying $[\omega]_{A_q} \leq M_2$,
\begin{equation}\label{bdry-cond-last}
\omega(\{x\in\mathbb{R}^{n}: \M_{\mu, \Omega_{2R}(y_0)}(|\nabla u|^{2})  > N \}) \leq \epsilon \omega(B_{r_0}(y)), \quad \forall \ y \in \overline{\Omega}_R(y_0),
\end{equation}
then with $\epsilon_1 = 5^{nq} \max\{A^{-1}, 4^n\}^{q} M_2^2\epsilon$, and for any $k \in \mathbb{N}$, 
\begin{equation} \label{iteration-formula}
\begin{split}
& \omega \Big  (\Big\{\Omega_R(y_0): \M_{\mu, \Omega_{2R}(y_0)}(|\nabla u|^2 > N^k \Big\} \Big)  \leq \epsilon_1^k \omega\Big (\Big\{\Omega_R(y_0): \M_{\mu, \Omega_{2R}(y_0)}(|\nabla u|^2 )> 1 \Big\} \Big) \\
& \quad \quad + \sum_{i=1}^k \epsilon_1^{i}\left[ \omega\Big (\Big\{\Omega_{R}(y_0): \M_{\mu, \Omega_{2R}(y_0)}(|\nabla g|^2) > \delta N^{k-i}  \Big\} \Big)+\omega\Big (\Big\{\Omega_{R}(y_0): \M_{\mu, \Omega_{2R}(y_0)}(|\F/\mu|^2) > \delta N^{k-i}\Big\} \Big) \right].
\end{split}
\end{equation}
\end{lemma}
\begin{proof} We use induction on $k$. If $k =1$, \eqref{iteration-formula}  holds as a result of Proposition \ref{decay-est-global}. Now, let us assume that Lemma \ref{global-iterating-lemma} holds for some $k \in \{1, 2,\cdots, k_0\}$ with some $k_0 \in \N$. Assume that $u$ is a weak solution of \eqref{lambda-eqn-all} with some $\lambda \geq \lambda_0$ so that \eqref{bdry-cond-last} holds. 
Now, let us define $u' = u/\sqrt{N}$, $\F' = \F/\sqrt{N}$, $g' = g/\sqrt{N}$, and $\lambda' = \lambda \sqrt{N} \geq \lambda \geq \lambda_0$. Then, we see that $u'$ is a weak solution of
\[
\left\{
\begin{array}{cccl}
\text{div}[\A_{\lambda'}(x,u', \nabla u')] & = & \text{div}[\F'], & \quad \text{in} \quad \Omega,\\
u' & =& g', & \quad \text{on} \quad\partial \Omega.
\end{array} \right.
\]
Moreover, $\forall \ y \in \overline{\Omega}_{R}(y_0)$
\[
\omega\Big(\Big\{\Omega_R(y_0): \M_{\mu, \Omega_{2R}(y_0)}(|\nabla u'|^2 ) > N  \Big\} \Big) = \omega\Big(\Big\{\Omega_R(y_0): \M_{\mu, \Omega_{2R}(y_0)}(|\nabla u|^2 ) > N^2  \Big\} \Big) \leq \epsilon \omega(B_{r_0}(y)).
\]
Then, by applying the the induction hypothesis on $u'$, we obtain
\[
\begin{split}
& \omega\Big (\Big\{\Omega_R(y_0):   \M_{\mu, \Omega_{2R}(y_0)}(|\nabla u'|^2 > N^{k_0} \Big\} \Big)  \\
& \leq \epsilon_1^{k_0} \omega\Big (\Big\{\Omega_R(y_0): \M_{\mu, \Omega_{2R}(y_0)}(|\nabla u'|^2 )> 1 \Big\} \Big) \\\
& + \sum_{i=1}^{k_0} \epsilon_1^{i}\left[ \omega\Big (\Big\{\Omega_R(y_0): \M_{\mu, \Omega_{2R}(y_0)}(|\nabla g'|^2) > \delta N^{k_0-i}  \Big\} \Big)+\omega\Big (\Big\{\Omega_R(y_0): \M_{\mu, \Omega_{2R}(y_0)}(|\F'/\mu|^2) > \delta N^{k_0-i}\Big\} \Big) \right].
\end{split}
\]
This implies
\[
\begin{split}
& \omega\Big (\Big\{\Omega_R(y_0):   \M_{\mu, \Omega_{2R}(y_0)}(|\nabla u|^2 > N^{k_0+1} \Big\} \Big) \\
& \leq \epsilon_1^{k_0} \omega\Big (\Big\{\Omega_R(y_0): \M_{\mu, \Omega_{2R}(y_0)}(|\nabla u|^2 )> N \Big\} \Big) \\\
& + \sum_{i=1}^{k_0} \epsilon_1^{i}\left[ \omega\Big (\Big\{\Omega_R(y_0): \M_{\mu, \Omega_{2R}(y_0)}(|\nabla g|^2) > \delta N^{k_0 +1 -i}  \Big\} \Big)+\omega\Big (\Big\{\Omega_R(y_0): \M_{\mu, \Omega_{2R}(y_0)}(|\F'/\mu|^2) > \delta N^{k_0 +1-i}\Big\} \Big) \right].
\end{split}
\]
On the other hand, we also have
\[
\begin{split}
& \omega\Big (\Big\{\Omega_R(y_0): \M_{\mu, \Omega_{2R}(y_0)}(|\nabla u|^2 > N \Big\} \Big) \\
& \leq \epsilon_1\left[ \omega\Big (\Big\{\Omega_R(y_0): \M_{\mu, \Omega_{2R}(y_0)}(|\nabla u|^2 )> 1 \Big\} \Big) \right.\\
& \quad + \left. \omega\Big (\Big\{\Omega_R(y_0): \M_{\mu, \Omega_{2R}(y_0)}(|\nabla g|^2) > \delta  \Big\} \Big)+\omega\Big (\Big\{\Omega_R(y_0): \M_{\mu, \Omega_{2R}(y_0)}(|\F/\mu|^2) > \delta \Big\} \Big) \right].
\end{split}
\]
Therefore, by combining the last two estimates, we see that
\[
\begin{split}
& \omega\Big (\Big\{\Omega_R(y_0):  \M_{\mu, \Omega_{2R}(y_0)}(|\nabla u|^2 > N^{k_0+1} \Big\} \Big) \\
& \leq \epsilon_1^{k_0+1} \omega\Big (\Big\{\Omega_R(y_0): \M_{\mu, \Omega_{2R}(y_0)}(|\nabla u|^2 )> N \Big\} \Big) \\\
& + \sum_{i=1}^{k_0+1} \epsilon_1^{i}\left[ \omega\Big (\Big\{\Omega_R(y_0): \M_{\mu, \Omega_{2R}(y_0)}(|\nabla g|^2) > \delta N^{k_0 +1 -i}  \Big\} \Big)+\omega\Big (\Big\{\Omega_R(y_0): \M_{\mu, \Omega_{2R}(y_0)}(|\F'/\mu|^2) > \delta N^{k_0 +1-i}\Big\} \Big) \right].
\end{split}
\]
Hence, we have proved Lemma \ref{global-iterating-lemma} for $k \leq k_0 +1$. Lemma \ref{global-iterating-lemma} then holds for all $k \in \N$ by induction, and the proof is complete.
\end{proof}
\noindent
{\bf Proof of Theorem \ref{main-theorem}}. Let $N = N(\Lambda, M_0, n)$ be defined as in Lemma \ref{global-iterating-lemma}. For $p >2$, we denote $s = p/2 >1$, and choose $\epsilon >0$ and sufficiently small and depending only on $\Lambda, A, M_0, n, p, q$ such that
\[
\epsilon_1 N^{s} = 1/2,
\]
where $\epsilon_1$ is defined in Lemma \ref{global-iterating-lemma}. With this $\epsilon$, we can now choose 
\[ 
\begin{split}
\delta = \delta(q, \Lambda, A, M_0, M_2,p, n), \quad \lambda_0 = \lambda_0(\Lambda, A, M_0, M_1, M_2, \omega_0, p, q, n)
 \end{split}
\]
as determined by Lemma \ref{global-iterating-lemma}.  For $\lambda \geq \lambda_0$, let us assume for a moment that $u $ is a weak solution of \eqref{lambda-eqn-all}. Let
\begin{equation} \label{E-def}
E =E(\lambda, N) = \Big\{\Omega_R(y_0): \M_{\mu, \Omega_{2R}(y_0)} (|\nabla u|^2) > N \Big\}.
\end{equation}
We can assume without loss of generality that $r_0 < \min\{R, R_0\}/50$. We first prove the estimate in Theorem \ref{main-theorem} with the extra condition that
\begin{equation} \label{extra}
\omega(E) \leq \epsilon \omega(B_{r_0}(y)), \quad \forall \ y \in \overline{\Omega}_R(y_0).
\end{equation}
Let us now consider the sum
\[
S = \sum_{k=1}^\infty N^{sk}  \omega\Big (\Big\{ \Omega_R(y_0): \M_{\mu, \Omega_{2R}(y_0)}(|\nabla u|^2) > N^k \Big\} \Big).
\]
From \eqref{extra}, we can apply Lemma \ref{global-iterating-lemma} to obtain
\[
\begin{split}
S & \leq \sum_{k=1}^\infty N^{ks} \left[ \sum_{i=1}^k \epsilon_1^i \omega\Big( \Big\{ \Omega_R(y_0): \M_{\mu, \Omega_{2R}(y_0)}(|\F/\mu|^2) > \delta N^{k-i} \Big\}\Big) \right. \\
& \quad \quad \quad \quad + \left. \sum_{i=1}^k \epsilon_1^i   \omega\Big( \Big\{ \Omega_R(y_0): \M_{\mu, \Omega_{2R}(y_0)}(|\nabla g|^2) >\delta N^{k-i} \Big\}\Big)   \right] \\
& \quad + \sum_{k=1}^\infty \Big(N^{s} \epsilon_1\Big)^k \omega\Big( \Big\{ \Omega_R(y_0): \M_{\mu, \Omega_{2R}(y_0)}(|\nabla u|^2) >1 \Big\}\Big).
\end{split}
\]
By Fubini's theorem, the above estimate can be rewritten as
\begin{equation} \label{Fubini-express}
\begin{split}
S &\leq \sum_{j=1}^\infty (N^s \epsilon_1)^j \left[  \sum_{k=j}^\infty N^{s(k-j)} \omega\Big( \Big\{ \Omega_R(y_0): \M_{\mu, \Omega_{2R}(y_0)}(|\F/\mu|^2) > \delta N^{k-j} \Big\}\Big)  \right. \\
&\quad \quad \quad + \left. \sum_{k=j}^\infty N^{s(k-j)}   \omega\Big( \Big\{ \Omega_R(y_0): \M_{\mu, \Omega_{2R}}(|\nabla g|^2) >\delta N^{k-j} \Big\}\Big)  \right] \\
& \quad + \sum_{k=1}^\infty \Big(N^{s} \epsilon_1\Big)^k \omega\Big( \Big\{ \Omega_R(y_0): \M_{\mu, \Omega_{2R}(y_0)}(|\nabla u|^2) >1 \Big\}\Big).
\end{split}
\end{equation}
Observe that 
\[
\omega\Big( \Big\{ \Omega_R(y_0): \M_{\mu, \Omega_{2R}(y_0)}(|\nabla u|^2) >1 \Big\}\Big) \leq\omega\Big(\Omega_R(y_0)\Big).
\]
From this, the choice of $\epsilon$, and Lemma \ref{measuretheory-lp}, and \eqref{Fubini-express} it follows that
\[
\begin{split}
S \leq C \left [ \norm{\M_{\mu, \Omega_{2R}(y_0)} ( |\F/\mu|^2)}_{L^s(\Omega_R(y_0), \omega)}^s + \norm{\M_{\mu,\Omega_{2R}(y_0)}(|\nabla g|^2)}_{L^s(\Omega_R(y_0), \omega)}^s + \omega(\Omega_R(y_0)) \right].
\end{split}
\]
Applying the Lemma \ref{measuretheory-lp} again, we see that
\[
\norm{\M_{\mu, \Omega_{2R}(y_0)}(|\nabla u|^2)}_{L^s(\Omega_R(y_0), \omega)}^s \leq C\left [ \norm{\M_{\mu, \Omega_{2R}(y_0)} ( |\F/\mu|^2)}_{L^s(\Omega_{2R}(y_0), \omega)}^s + \norm{\M_{\mu, \Omega_{2R}(y_0)}(|\nabla g|^2)}_{L^s(\Omega_{2R}(y_0), \omega)}^s + \omega(\Omega_R(y_0))  \right].
\]
By the Lesbegue's differentiation theorem, we observe that
\[
|\nabla u(x)|^2 \leq \M_{\mu,\Omega_{2R}(y_0)}(|\nabla u|^2)(x), \quad \text{a.e} \ x\ \in \Omega_R(y_0).
\]
Hence,
\[
\norm{\nabla u}_{L^p(\Omega_R(y_0), \omega)}^p \leq C\left[ \norm{\M_{\mu, \Omega_{2R}(y_0)} (|\F/\mu|^2)}_{L^s(\Omega_R(y_0), \omega)}^s + \norm{\M_{\mu,\Omega_{2R}(y_0)}(|\nabla g|^2)}_{L^s(\Omega_R(y_0), \omega)}^s + \omega(\Omega_R(y_0))  \right].
\]
Since $(\mu, \omega)$ satisfies the $s$-Sawyer's condition, by Theorem \ref{Two-weighted-maximal-ineq}, it follows
\begin{equation} \label{ineq-extra-cond}
\norm{\nabla u}_{L^p(\Omega_R(y_0), \omega)} \leq C\left[\norm{\F/\mu}_{L^p(\Omega_{2R}(y_0), \omega)} + \norm{\nabla g}_{L^p(\Omega_{2R}(y_0), \omega)} + \omega(\Omega_R(y_0))^{1/p} \ \right].
\end{equation}
Thus, we have proved \eqref{ineq-extra-cond} as long as $u$ is a weak solution of \eqref{lambda-eqn-all} for all $\lambda \geq \lambda_0$ and \eqref{extra} holds. \\

Next, we consider the case $0 < \lambda < \lambda_0$. Assume that  $u$ is a weak solution of \eqref{lambda-eqn-all} and \eqref{extra} holds.  Let us denote $u' = u/(\lambda_0/\lambda), \F' = \F/(\lambda_0/\lambda), g' = g/(\lambda_0/\lambda)$. Then, $u'$ is a weak solution of
\[
\left\{
\begin{array}{cccl}
\text{div}[\A_{\lambda_0}(x, u', \nabla u')] & = & \text{div}[\F'], & \quad \text{in} \quad \Omega, \\
u' & = & g', & \quad \text{on} \quad \Omega.
\end{array}
\right.
\]
Moreover, because of \eqref{extra} and $\lambda_0/\lambda \geq 1$, we also have
\[
\omega\Big(\Omega_R(y_0): \Big\{ \M_{\mu, \Omega_{2R}(y_0)} (|\nabla u'|^2) >N \Big\}\Big) \leq \epsilon \omega(B_{r_0}(y)), \quad \forall \ y \in \overline{\Omega}_R(y_0).
\]
Therefore, applying the conclusion of \eqref{ineq-extra-cond} for $u'$, we also obtain
\[
\norm{\nabla u'}_{L^p(\Omega_R(y_0), \omega)} \leq C\left[\norm{\F'/\mu}_{L^p(\Omega_{2R}(y_0), \omega)} + \norm{\nabla g'}_{L^p(\Omega_{2R}(y_0), \omega)} + \omega(\Omega_R(y_0))^{1/p}  \right].
\]
Thus, 
\begin{equation*} 
\norm{\nabla u}_{L^p(\Omega_R(y_0), \omega)} \leq C\left[\norm{\F/\mu}_{L^p(\Omega_{2R}(y_0), \omega)} + \norm{\nabla g}_{L^p(\Omega_{2R}(y_0), \omega)} + \lambda_0\omega(\Omega_R(y_0))^{1/p}/\lambda   \right].
\end{equation*}
In summary, up to now, we have proved that if $u$ is a weak solution of \eqref{lambda-eqn-all} with $\lambda>0$ and if \eqref{extra} holds, then 
\begin{equation} \label{ineq-extra-cond-2}
\norm{\nabla u}_{L^p(\Omega_R(y_0), \omega)} \leq C\left[\norm{\F/\mu}_{L^p(\Omega_{2R}(y_0), \omega)} + \norm{\nabla g}_{L^p(\Omega_{2R}(y_0), \omega)} + \omega(\Omega_R(y_0))^{1/p}\max\{\lambda^{-1}, 1\}   \right].
\end{equation}
We now remove the extra condition \eqref{extra}. Assuming now that $u$ is a weak solution of \eqref{main-eqn}, i.e. $\lambda =1$. 
Let $M >1$ sufficiently large and will be determined. Let $u_M = u/M, \F_M = \F/M$ and $g_M = g/M$. We note that $u_M$ is a weak solution of
\begin{equation} \label{u-M-eqn}
\left\{
\begin{array}{cccl}
\text{div}[\A_M(x, u_M, \nabla u_M)] & = &\text{div}[\F_M], & \quad \text{in} \quad \Omega, \\
u_M & = & g_M, & \quad \text{on} \quad \partial \Omega.
\end{array}
\right.
\end{equation}
Let us denote
\[
E_M = \Big\{\Omega_R(y_0): \M_{\mu, \Omega_{2R}(y_0)}(|\nabla u_M|^2) > N \Big\}.
\]
and
\begin{equation} \label{K-zero}
K_0 = \left(\frac{1}{\mu(B_{2R}(y_0))} \int_{\Omega_{2R}(y_0)} |\nabla u|^2 \mu(x) dx \right)^{1/2} +1 \geq 1.
\end{equation}
We claim that we can choose $M  = C K_0$ with some sufficiently large constant $C$ depending only on $\Lambda, M_0, M_1$, $M_2, p, q, n$ and $R/r_0$ such that
\begin{equation} \label{M-density}
\omega(E_M) \leq \epsilon \omega(B_{r_0}(y)), \quad \forall \ y \in \overline{\Omega}_R(y_0).
\end{equation}
If this holds, we can apply \eqref{ineq-extra-cond-2} for $u_M$ which is a weak solution of \eqref{u-M-eqn} to obtain
\[
\norm{\nabla u_M}_{L^p(\Omega_{R}(y_0), \omega)} \leq C \left[ \norm{\F_M/\mu}_{L^p(\Omega_{2R}(y_0), \omega)}  + \norm{g_M}_{L^p(\Omega_{2R}(y_0), \omega)} +  \omega(\Omega_{R}(y_0))^{1/p} \right].
\]
Then, by multiplying this equality with $M$, we obtain
\[
\norm{\nabla u}_{L^p(\Omega_R(y_0), \omega)} \leq C \left[ \norm{\F/\mu}_{L^p(\Omega_{2R}(y_0), \omega)}  + \norm{g}_{L^p(\Omega_{2R}(y_0), \omega)} + \omega(\Omega_{R}(y_0))^{1/p} K_0\right].
\]
The proof of Theorem \ref{main-theorem} is therefore complete if we can prove \eqref{M-density}.  To this end, using the doubling property of $\omega \in A_q$, Lemma \ref{doubling}, we have
\[
\begin{split}
\frac{\omega(E_M)}{\omega(B_{r_0}(y))} & = \frac{\omega(E_M)}{\omega(B_{2R}(y_0))} \frac{\omega(B_{2R}(y_0))}{\omega(B_{r_0}(y))}  \leq M_2 \frac{\omega(E_M)}{\omega(B_{2R}(y_0))}  \left(\frac{2R}{r_0} \right)^{nq}.
\end{split}
\]
From this, and using Lemma \ref{doubling} again, we can find $\beta = \beta(M_2,n)>0$ such that
\begin{equation} \label{omega-mu-comparision}
\frac{\omega(E_M)}{\omega(B_{r_0}(y))} \leq C(M_0, M_2, n)  \left(\frac{2R}{r_0} \right)^{nq} \left(\frac{\mu(E_M)}{\mu(B_{2R}(y_0))}\right)^{\beta/2}.
\end{equation}
Now, by the definition of $E_M$, and the weak type (1-1) estimate for maximal function, we see that
\[
\begin{split}
\frac{\mu(E_M)}{\mu(B_{2R}(y_0))} & = \mu\Big(\Big\{\Omega_{R}(y_0): \M_{\mu, \Omega_{2R}(y_0)}(|\nabla u|^2) > NM^2 \Big\} \Big) \\
& \leq \frac{C(M_0, n)}{N M^2}  \frac{1}{\mu(B_{2R}(y_0))}\int_{\Omega_{2R}(y_0)} |\nabla u|^2 \mu(x) dx \\
& \leq \frac{C(M_0, n) K_0^2}{NM^2},
\end{split}
\]
where  $K_0$ is defined in \eqref{K-zero}. 
From this, and \eqref{omega-mu-comparision}, it follows that
\[
\frac{\omega(E_M)}{\omega(B_{r_0}(y))} \leq  C^*(\Lambda, M_0, M_2, n)  \left(\frac{2R}{r_0} \right)^{n q} \left(\frac{ K_0}{M}\right)^\beta
\]
Now, we choose $M$ such that
\[
M = K_0\left[\epsilon^{-1} C^*(\Lambda, M_0, M_2, n)  \left(\frac{2R}{r_0} \right)^{nq}\right]^{1/\beta}
\]
then, it follows
\[
\omega(E_M) \leq \epsilon \omega(B_{r_0}(y)), \quad \forall \ y \in \overline{\Omega}_R(y_0).
\]
This proves \eqref{M-density} and completes the proof.

\appendix  \label{Appendix}
\section{Proof of Lemma \ref{Vitali}}
\begin{proof}
We follow the method used in \cite{BW2, MP-1} with some modifications fitting to our setting. For each $x \in C$, let us define
\[
\phi(\rho) = \frac{\omega(C\cap B_\rho(x))}{\omega(B_\rho(x))}.
\]
By the Lebesgue Differentiation Theorem, for almost every $x \in C$, $\phi$ is continuous and $\phi(0) = \displaystyle{\lim_{\rho\rightarrow 0^+} \phi(\rho)} = 1$. Moreover, by (i), $\phi(r_0) < \epsilon$. Therefore, for almost every $x \in C$, there is $0 < \rho_x < r_0$ such that
\begin{equation} \label{cover-balls}
\begin{split}
\omega(C \cap B_{\rho_x}(x)) & = \epsilon \omega (B_{\rho_x}(x)), \quad \text{and} \\
\omega(C\cap B_\rho(x)) & < \epsilon \omega (B_\rho(x)), \quad \rho > \rho_x.
\end{split}
\end{equation}
Now, observe that the family of balls $\{B_{\rho_x}(x)\}_{x \in C}$ covers $C$. Therefore, by Vitali's Covering Lemma, there exists a countable $\{x_k\}_{k \in \N}$ in $C$ such that the balss $\{B_{\rho_{k}}(x_k)\}_{k \in \N}$ with $\rho_k = \rho_{x_k}$ are disjoint and
\[
C \subset \displaystyle{\cup_{k=1}^\infty B_{5\rho_k}(x_k)}.
\]
From this, \eqref{cover-balls} and Lemma \ref{doubling}, we infer that
\[
\begin{split}
\omega(C) & \leq \omega\Big (C\cap \big(\cup_{k=1}^\infty B_{5\rho_k}(x_k) \big) \Big ) \leq \sum_{k=1}^N \omega\Big (C\cap B_{5\rho_k}(x_k) \Big )  \\
& < \epsilon \sum_{k=1}^N\omega (B_{5\rho_k}(x_k)) \leq \epsilon M 5^{np}\sum_{k=1}^N \omega(B_{\rho_k}(x_k)).
\end{split}
\]
Observe that by (i), and since $B_{\rho_k}(x_k)$ are all disjoint,
 \[ \Big( \cup_{k=1}^n B_{\rho_k}(x_k)  \Big) \cap \Omega_R(y_0) \subset D. \]
We claim that  
\begin{equation} \label{A-type}
\omega(B_{\rho}(x)) \leq \max M\Big\{A^{-1}, 4^n \Big\}^q \omega(B_{\rho}(x) \cap \Omega_R(y_0)), \quad \forall  x \in \Omega_R(y_0), \ \rho \in (0, r_0).
\end{equation}
From this claim, it follows that
\[
\omega(C) \leq \epsilon' \sum_{k=1}^n \omega\Big(B_{\rho_k}(x_k) \cap \Omega_R(y_0)\Big) = \epsilon' \omega \Big( \Big( \cup_{k=1}^n B_{\rho_k}(x_k)  \Big) \cap \Omega_R(y_0)\Big) \leq \epsilon' \omega(D).
\]
It now remains to prove \eqref{A-type}. It follows from Lemma \ref{doubling} that
\[
\frac{\omega(B_\rho(x))}{\omega(B_\rho(x) \cap \Omega_R(y_0))} \leq M \left(\frac{|B_\rho(x)|}{|B_\rho(x) \cap \Omega_R(y_0)|} \right)^q.
\] 
Hence, it suffices to prove
\begin{equation} \label{ball-intersection-eqn}
\sup_{x \in \Omega_R(y_0)}\sup_{0 <\rho < r_0} \frac{|B_\rho(x)|}{|B_\rho(x) \cap \Omega_R(y_0)|} \leq \Big\{ A^{-1}, 4^n \Big\}.
\end{equation}
We fix $x \in \Omega_R(y_0)$ and $\rho \in (0,r_0)$. Observe that if $B_{\rho}(x) \subset \Omega_{R}(y_0)$. Then \eqref{ball-intersection-eqn} is trivial.  Hence, we only need to consider the case $B_{\rho}(x) \cap \partial \Omega_{R}(y_0) \not= \emptyset$. We divide this situation into three subcases.\\
{\bf Case I:} If $B_{\rho}(x) \cap \Big( \partial B_R(y_0) \cap \overline{\Omega} \Big) \not=\emptyset$ and $B_{\rho}(x) \cap \partial \Omega \cap \overline{B}_R(y_0) = \emptyset$. Then, it follows that $\Omega_R(y_0) \cap B_{\rho}(x) = B_R(y_0) \cap B_{\rho}(x)$. From this, a simple calculation shows
\[
\frac{|B_\rho(x)|}{|B_\rho(x) \cap \Omega_R(y_0)|} = \frac{|B_\rho(x)|}{|B_\rho(x) \cap B_R(y_0)|} \leq 4^n.
\]
\\
\noindent
{\bf Case II:} If $B_{\rho}(x) \cap \partial \Omega   \cap \overline{B}_R(y_0) \not= \emptyset$ and $B_{\rho}(x) \cap  \partial B_R(y_0) \cap \overline{\Omega} =\emptyset$. In this case, by the proof of \cite[Lemma 3.8]{MP-1}, it follows that
\[
\frac{|B_\rho(x)|}{|B_\rho(x) \cap \Omega_R(y_0)|} \leq \left(\frac{2}{1-4\delta} \right)^n \leq 4^n.
\]
\noindent
{\bf Case III:} If $B_{\rho}(x) \cap \partial \Omega \cap \overline{B}_R(y_0) \not= \emptyset$ and $B_{\rho}(x) \cap \partial B_R(y_0) \cap \overline{\Omega}  \not=\emptyset$. Then, by the definition of 
type $(A, r_0)$ domain, we see that
\[
|\Omega_R(y_0) \cap B_\rho(x)| \geq A|B_\rho(x)|.
\]
Therefore,
\[
\frac{|B_\rho(x)|}{|\Omega_R(y_0) \cap B_\rho(x)|} \leq A^{-1}.
\]
This completes the proof.
\end{proof}
 \ \\ \noindent
\textbf{Acknowledgement.}  T. Phan's research is supported by the Simons Foundation, grant \#~354889.

\end{document}